\documentclass[review,onefignum,onetabnum]{siamart190516}

\usepackage{lipsum}
\usepackage{amsfonts}
\usepackage{graphicx}
\usepackage{subfigure}
\usepackage{epstopdf}
\usepackage{algorithmic}
\ifpdf
  \DeclareGraphicsExtensions{.eps,.pdf,.png,.jpg}
\else
  \DeclareGraphicsExtensions{.eps}
\fi

\renewcommand\theequation{\thesection.\arabic{equation}}
\theoremstyle{plain}
\newtheorem{thm}{Theorem}[section]
\newtheorem{lem}{Lemma}[section]

\newtheorem{prop}{Proposition}[section]
\newtheorem{remark}{\textbf{Remark}}[section]

\newcommand{\eps}{\epsilon}

\newcommand{\bm}{\boldsymbol}

\newcommand{\Grad}[1]{\nabla #1}

\newcommand{\ph}{\phi}

\newcommand{\be}{\begin{equation}}
\newcommand{\ee}{\end{equation}}

\newcommand{\bse}{\begin{subequations}}
\newcommand{\ese}{\end{subequations}}
\def\benl{\begin{eqnarray*}}
\def\eenl{\end{eqnarray*}}

\def\bphi{\bm{\phi}}

\def\be{\bm{e}}

\def\bx{\bm{x}}

\def\bmu1{\bm{\mu_1}}

\newcommand{\ben}{\begin{eqnarray}}
\newcommand{\een}{\end{eqnarray}}
\newcommand{\beq}{\begin{equation}}
\newcommand{\eeq}{\end{equation}}
\newcommand{\bea}{\begin{array}}
\newcommand{\eea}{\end{array}}
\newcommand{\bef}{\begin{figure}[H]}
\newcommand{\eef}{\end{figure}}

\newsiamremark{hypothesis}{Hypothesis}
\crefname{hypothesis}{Hypothesis}{Hypotheses}
\newsiamthm{claim}{Claim}

\title{Global constraints preserving SAV schemes  for   gradient flows}

\author{Qing Cheng\thanks{Department of Applied Mathematics, Illinois Institute of Technology, Chicago, IL 60616, USA 
  (\email{qcheng4@iit.edu)}. }
\and Jie Shen\thanks{Department of Mathematics, Purdue University, West Lafayette, IN 47907, USA  (shen7$@$purdue.edu)}. The work of J.S. is supported in part by NSF  DMS-1720442 and NFSC 11971407.}
\usepackage{amsopn}

\begin{document}
\bibliographystyle{plain}
\graphicspath{ {Figures/} }
\maketitle

\begin{abstract}
  We develop several efficient numerical schemes which  preserve exactly the global constraints  for constrained  gradient flows. Our schemes are based on the SAV approach combined with the Lagrangian multiplier approach. They are as efficient as the SAV schemes for unconstrained gradient flows, i.e., only require solving linear  equations with constant coefficients at each time step plus an additional nonlinear algebraic system which can be solved at negligible cost,  can be unconditionally energy stable, and preserve exactly the global constraints  for constrained gradient flows.    Ample numerical results for phase-field vesicle membrane and optimal partition models are presented  to validate the effectiveness and accuracy of the proposed numerical schemes.
\end{abstract}

\begin{keywords}
 constrained gradient flow, SAV approach, energy stability, phase-field
\end{keywords}


\section{Introduction}
Gradient flows are ubiquitous in science and engineering applications. In the last few decades, a large body  of research has been devoted to developing efficient numerical schemes, particularly time discretization schemes,  for gradient flows.  We refer to the recent review paper \cite{Du.F19}  and the references therein,  for a detailed account on these efforts, see also \cite{SXY19,She.Y19} for 
a presentation of the newly developed  IEQ method \cite{yang2016linear,YANG2017_3phase} and  SAV method (cf. \cite{SXY19,cheng2018multiple})  which have received much attention recently due to their efficiency, flexibility and accuracy.
However, most of the research in this area are concerned with unconstrained gradient flows.
 But many gradient flows  in physical, chemical and biological sciences are  constrained with one or several global  physical constraints, e.g.,  the norm of  multi-component wave functions is preserved in  multi-component Bose-Einstein condensates \cite{bao2004computing},   the norm  of each components is preserved in optimal eigenvalue partition problems \cite{CL1,conti2003optimal},  the stress is  constrained to be negative in  topology optimization problem \cite{stress2016,jeong2014development},  the volume \cite{jing2019second,du2012analysis}  and surface area are preserved in the  phase field vesicle membrane model \cite{du2008numerical,cheng2018multiple,wang2016efficient,guillen2018unconditionally,du2009energetic}, and many others in constrained minimizations.

A highly desirable property of numerical algorithms for gradient flows with physical  constraints is to be able to satisfy  the energy dissipation law  while  preserving  the
physical constraints.  But
how to design numerical schemes which are energy stable while enforcing   physical constraints, such as mass, norm or surface area conservation, is a challenging task. There are essentially three different approaches:
\begin{itemize}
 \item Direct discretization: discretize the constraints along with the gradient flows  for a given order of accuracy, see Section 2.1 for  example. This is a straightforward approach that can be coupled with existing efficient numerical schemes for gradient flows that can be easily implemented.  But its drawback is that the constraints can only be satisfied up to the order of accuracy (usually first- or second-order).
 
 \item Penalty approach: add  suitable penalty terms in the free energy and consider the unconstrained gradient flow with the new penalized free energy.  Its advantage is that efficient numerical methods for unconstrained gradient flows can be directly applied,  and in principle one can approximate the constraints to within arbitrary accuracy by choosing suitable penalty parameters.  Its disadvantage is that large penalty parameters,  which are needed for more accurate approximation of the constraints,  may lead to very stiff systems that are difficult to solve efficiently.  This approach is used in \cite{wang2016efficient,yang2017efficient,cheng2018multiple} for the vesicle membrane model and in \cite{zhuang2019efficient} for the multi-component Bose-Einstein condensates.
 
  \item  Lagrangian multiplier approach: introduce Lagrangian multipliers to enforce exactly the constraints.  This approach is studied mathematically in \cite{Caf.L09}  and numerical in \cite{du2008numerical}  for the optimal eigenvalue partition problem.  The main advantage is that the constraints can be satisfied exactly,  while its drawback is that it will lead to difficulty  to solve nonlinear systems  at each time step. 
\end{itemize}
The goal of this paper is to develop efficient  time discretization for gradient flows with global  constraints  using  the Lagrangian multiplier approach.  To this end,  we shall combine the SAV approach \cite{SXY19} with the Lagrangian multiplier approach \cite{cheng2019new}., hoping to construct schemes which enjoy all advantages of the SAV schemes,  but can also preserve the constraints exactly  using the Lagrangian multiplier approach with negligible additional cost.  Three different approaches will be considered: (i)  The first one is a direct combination of  the SAV approach with the Lagrangian multiplier approach.  The scheme is easy to implement but  we are unable to prove that it is unconditionally energy stable.  (ii) In the second approach,  we replace the dynamic equation for the SAV by another Lagrangian multiplier so that the scheme becomes unconditionally energy stable,  but leading to an additional coupled nonlinear algebraic system for the two Lagrangian multipliers to solve at each time step,  instead of  a nonlinear algebraic equation for just one  Lagrangian multiplier in the first approach.  (iii) In the third approach,  we combine the advantages of the first two approaches to construct a scheme,  which is unconditionally energy stable,  such that the two  Lagrangian multipliers  can be determined sequentially,  instead of a coupled system in the second approach.  All three approaches have essentially the same computational cost as the linear SAV scheme,  presented in Section 2.1,  which is extremely efficient and easy to implement but can only approximate the constraints up to the order of the scheme.
Our numerical results indicate that the first and third approaches  are  generally more efficient and robust than second approach.

The reminder of this paper is structured as follows.  In Section 2,  we  present a general methodology for gradient flows with one global constraint,  and  propose three different approaches to devise efficient numerical schemes which can enforce exactly the constraint.  In Sections 3 and 4,  we apply the general methodology introduced in Section 2
 to the phase field vesicle membrane model with two constraints  and  to the optimal partition model with multiple constraints,  respectively.  In Section 5,  we present numerical experiments to compare the performance of  the three approaches and the penalty approach,  and present several  numerical simulations for  the phase field vesicle membrane model with two constraints  and  to the optimal partition model with multiple constraints.  Some concluding remarks are given in Section 6.

\section{General methodology to preserve global constraints for gradient flows}
\label{sec:main}
We present in this section  a general methodology to preserve global constraints in gradient flows.  To simplify the presentation,  we consider here single-component gradient flows with  a single global constraint.  The  approaches developed here will be extended to problems with multi-components and multi global constraints  in the subsequent sections. 

To fix the idea,  we consider a system with total free energy  in the form 
\begin{equation}\label{orienergy}
 E(\phi)=\int_\Omega \frac12\mathcal{L}\phi\cdot \phi +F(\phi) d\bx,
\end{equation}
under a global constraint 
\begin{equation}
 \frac{d}{dt} H(\phi) =0 \;\text{ with }\;  H(\phi)=\int_\Omega h(\phi)d\bx,
\end{equation}
where $\mathcal{L}$ is certain linear positive operator,  $F(\phi)$ is a nonlinear potential,  and $h(\phi)$ is a function of $\phi$.  Then, by introducing a Lagrange multiplier $\lambda(t)$,   a general gradient flow with the above free energy under the constraint  takes the following form
\begin{equation}\label{grad:flow}
\begin{split}
&\phi_t=-\mathcal{G}\mu,\\
&\mu=\mathcal{L}\phi+F'(\phi)-\lambda \frac{\delta H}{\delta \phi},\\
& \frac{d }{d t} H(\phi) =0,
\end{split}
\end{equation}
where  $\mathcal{G}$  is  a  positive operator describing the relaxation process of the system. 
The boundary conditions can be either one of the following two type
\begin{eqnarray}
&&(i)\mbox{ periodic; or } (ii)\,\,\partial_{\bf n} \phi|_{\partial\Omega}=\partial_{\bf n}\mu|_{\partial\Omega}=0,
\end{eqnarray}
where $\bf n$ is the unit outward normal on the boundary $\partial\Omega$. Taking the inner products of the first two equations with $\mu$ and $-\phi_t$ respectively, summing up the results  using the third equation,  taking integration by part, we obtain the following energy dissipation law:
\begin{equation}
 \frac{d }{d t} E(\phi)=-(\mathcal{G}\mu,\mu),
\end{equation}
where, and in the sequel, $(\cdot, \cdot)$ denotes the inner product in $L^2(\Omega)$. We shall also denote the $L^2$-norm by $\|\cdot\|$. 

We shall first construct a linear scheme based on the SAV approach which only approximate the global constraint,  followed  by  three "essentially" linear   schemes which enforce exactly the global constraint while retaining all essential advantages of the SAV approach.

\subsection{A linear scheme based on the SAV approach}

We start by constructing  first  a linear scheme based on the SAV approach for  \eqref{grad:flow}.
Assuming $\int_{\Omega}F(\phi)d\bx>-C_0$ for some $C_0>0$,  we introduce a scalar auxiliary variable (SAV)  $r(t)=\sqrt{\int_{\Omega}F(\phi)d\bx+C_0}$ and rewrite \eqref{grad:flow} as 
\begin{eqnarray}
&&\partial_t\phi=-\mathcal{G}\mu,\label{sav:1}\\
&&\mu=\mathcal{L}\phi+\frac{r(t)}{\sqrt{\int_{\Omega}F(\phi)d\bx+C_0} }F'(\phi)-\lambda \frac{\delta H}{\delta \phi},\label{sav:2}\\
&&\frac{d }{d t} H(\phi) =0,\label{sav:3}\\
&&r_t=\frac 1{2\sqrt{\int_{\Omega}F(\phi)d\bx+C_0}}({F'(\phi)},\phi_t).\label{sav:4}
\end{eqnarray}
Taking the inner products of the first two equations with $\mu$ and $-\phi_t$ respectively,   summing up the results along with the fourth equation and using the third equation,  we obtain the following energy dissipation law:
\begin{equation}
\frac{d }{dt} \tilde E(\phi)=-(\mathcal{G}\mu,\mu),
\end{equation}
where $\tilde E(\phi)=\int_\Omega \frac12\mathcal{L}\phi\cdot \phi d\bx +r^2$ is a modified energy.
Then, a first-order SAV scheme for the above modified system is
\begin{eqnarray}
&&\frac{\phi^{n+1}-\phi^n}{\delta t}=-\mathcal{G}\mu^{n+1},\label{linear:sav:1b}\\
&&\mu^{n+1}=\mathcal{L}\phi^{n+1}+\frac{r^{n+1}}{\sqrt{\int_{\Omega}F(\phi^n)d\bx+C_0} }F'(\phi^n)-\lambda^{n+1} (\frac{\delta H}{\delta \phi})^n,\label{linear:sav:2b}\\
&&((\frac{\delta H}{\delta \phi})^n,\phi^{n+1}-\phi^n)=0,\label{linear:sav:3b}\\
&&\frac{r^{n+1}-r^n}{\delta t}=\frac 1{2\sqrt{\int_{\Omega}F(\phi^n)d\bx+C_0}}({F'(\phi^n)},\frac{\phi^{n+1}-\phi^n}{\delta t}).\label{linear:sav:4b}
\end{eqnarray}
Taking the inner products of \eqref{linear:sav:1b} with $\mu^{n+1}$ and of \eqref{linear:sav:2b} with $-\frac{\phi^{n+1}-\phi^n}{\delta t}$,  summing up the results and taking into account \eqref{linear:sav:3b}-\eqref{linear:sav:4b},  we have the following:
\begin{thm}
 The  scheme \eqref{linear:sav:1b}-\eqref{linear:sav:4b} is unconditionally energy stable in the sense that
 $$\tilde E(\phi^{n+1})-\tilde E(\phi^{n}) \le -\Delta t(\mathcal{G}\mu^{n+1},\mu^{n+1}),$$
 where $\tilde E(\phi^k)=\int_\Omega \frac12\mathcal{L}\phi^k\cdot \phi^k d\bx +(r^k)^2$.
\end{thm}
We now show that the above scheme can be efficiently implemented.
Writing 
\begin{equation}\label{pmr}
\phi^{n+1}=\phi_1^{n+1}+\lambda^{n+1}\phi_2^{n+1},\; \mu^{n+1}=\mu_1^{n+1}+\lambda^{n+1}\mu_2^{n+1},\; r^{n+1}=r_1^{n+1}+\lambda^{n+1}r_2^{n+1}
 \end{equation}
in the above, we find that $(\phi_i^{n+1},\mu_i^{n+1},r_i^{n+1})\;(i=1,2)$ can be determined  as follows:
\begin{eqnarray}
&&\frac{\phi_1^{n+1}-\phi^n}{\delta t}=-\mathcal{G}\mu_1^{n+1},\label{sav:1c}\\
&&\mu_1^{n+1}=\mathcal{L}\phi_1^{n+1}+\frac{r_1^{n+1}}{\sqrt{\int_{\Omega}F(\phi^n)d\bx+C_0} }F'(\phi^n),\label{sav:2c}\\
&&\frac{r_1^{n+1}-r^n}{\delta t}=\frac 1{2\sqrt{\int_{\Omega}F(\phi^n)d\bx+C_0}}({F'(\phi^n)},\frac{\phi_1^{n+1}-\phi^n}{\delta t});\label{sav:4c}
\end{eqnarray}
and
\begin{eqnarray}
&&\frac{\phi_2^{n+1}}{\delta t}=-\mathcal{G}\mu_2^{n+1},\label{sav:1d}\\
&&\mu_2^{n+1}=\mathcal{L}\phi_2^{n+1}+\frac{r_2^{n+1}}{\sqrt{\int_{\Omega}F(\phi^n)d\bx+C_0} }F'(\phi^n)- (\frac{\delta H}{\delta \phi})^n,\label{sav:2d}\\
&&\frac{r_2^{n+1}}{\delta t}=\frac 1{2\sqrt{\int_{\Omega}F(\phi^n)d\bx+C_0}}({F'(\phi^n)},\frac{\phi_2^{n+1}}{\delta t}).\label{sav:4d}
\end{eqnarray}
Since $r_i^{n+1}\, (i=1,2)$ is just a constant which can be easily eliminated by using a block Gaussian elimination, each of the above solutions  can be obtained  by solving two linear systems with constant coefficients of the form (cf. \cite{SXY19} for more detail):
  \begin{equation}\label{matrix2}
   \begin{pmatrix} 
      \frac1{\Delta t} I  &  \mathcal{G}  \\
      \mathcal{L}  & -I  
     \end{pmatrix}   
    \begin{pmatrix} 
      \phi \\
      \mu
         \end{pmatrix} 
   =\bar b.
    \end{equation}
Once we determine $(\phi_i^{n+1},\mu_i^{n+1},r_i^{n+1})\;(i=1,2)$ from the above,  we use \eqref{linear:sav:3b} to determine $\lambda^{n+1}$ explicitly by
\begin{equation}
 \lambda^{n+1}=-((\frac{\delta H}{\delta \phi})^n,\phi_1^{n+1}-\phi^n)/((\frac{\delta H}{\delta \phi})^n, \phi_2^{n+1}).
\end{equation}
Hence, the scheme is very efficient.  However,  the global constraint \eqref{sav:3} is only approximated to first-order.  While we can easily construct second-order energy stable SAV schemes which approximate \eqref{sav:3} to second-order,  we can not preserve \eqref{sav:3} exactly. 
Below,  we show how to modify the scheme \eqref{linear:sav:1b}-\eqref{linear:sav:4b} so that we can preserve \eqref{sav:3} exactly while keeping  its essential advantages.

\subsection{The first approach}
The first-approach is simply to replace the first-order approximation of  \eqref{linear:sav:3b}  by enforcing exactly  \eqref{sav:3}. More precisely,  a modified first-order scheme is as follows:
\begin{eqnarray}
&&\frac{\phi^{n+1}-\phi^n}{\delta t}=-\mathcal{G}\mu^{n+1},\label{sav:1b}\\
&&\mu^{n+1}=\mathcal{L}\phi^{n+1}+\frac{r^{n+1}}{\sqrt{\int_{\Omega}F(\phi^n)d\bx+C_0} }F'(\phi^n)-\lambda^{n+1} (\frac{\delta H}{\delta \phi})^n,\label{sav:2b}\\
&& H(\phi^{n+1}) = H(\phi^{0}),\label{sav:3b}\\
&&\frac{r^{n+1}-r^n}{\delta t}=\frac 1{2\sqrt{\int_{\Omega}F(\phi^n)d\bx+C_0}}({F'(\phi^n)},\frac{\phi^{n+1}-\phi^n}{\delta t}).\label{sav:4b}
\end{eqnarray}
The above scheme can be implemented in essentially the same procedure as the scheme \eqref{linear:sav:1b}-\eqref{linear:sav:4b}.  Indeed,  still writing $(\phi^{n+1},\mu^{n+1})$ as in \eqref{pmr}, we can still determine
$(\phi_i^{n+1},\mu_i^{n+1},\;i=1,2)$ from \eqref{sav:1c}-\eqref{sav:4c} and \eqref{sav:1d}-\eqref{sav:4d}.  The only difference is that we now need to determine
$\lambda^{n+1}$ from \eqref{sav:3b} which leads to a nonlinear algebraic equation for $\lambda^{n+1}$:
\begin{equation}\label{nonl1}
 (h(\phi_1^{n+1}+\lambda^{n+1}\phi_2^{n+1})-h(\phi^n),1)=0.
\end{equation}
The complexity of this nonlinear algebraic equation depends on $h(\phi)$, e.g.,  it will be a quadratic equation if $h(\phi)=\phi^2$ as in some applications. 
\begin{remark}
 The nonlinear algebraic equation \eqref{nonl1}  can   be solved by a Newton iteration whose convergence  depends on a good initial guess.  We can use the linear scheme \eqref{linear:sav:1b}-\eqref{linear:sav:4b}  to produce a good and reliable initial guess so that the Newton iteration will converge very quickly  with negligible cost. 
Hence, the system \eqref{sav:1b}-\eqref{sav:4b} is "essentially" linear as it involves a linear system plus a nonlinear algebraic equation,  and can be efficiently solved. 
\end{remark}

Next, we examine the stability of scheme  \eqref{sav:1b}-\eqref{sav:4b}. Taking the inner products of \eqref{sav:1b} with $\mu^{n+1}$,  of 
\eqref{sav:2b} with $-\frac{\phi^{n+1}-\phi^n}{\delta t}$ and of \eqref{sav:3b} with $2r^{n+1}$,  summing up the results,  we obtain
\begin{equation}
\begin{split}
 \frac1{2\delta t}\{(\mathcal{L}\phi^{n+1},\phi^{n+1})-& (\mathcal{L}\phi^{n},\phi^{n})+ (\mathcal{L}(\phi^{n+1}-\phi^n),\phi^{n+1}-\phi^n)\}\\
 &+\frac 1{\delta t}\{(r^{n+1})^2-(r^n)^2+(r^{n+1}-r^n)^2\}\\
 &=-(\mathcal{G}\mu^{n+1},\mu^{n+1})+\lambda^{n+1}((\frac{\delta H}{\delta \phi})^n, \frac{\phi^{n+1}-\phi^n}{\delta t}).
 \end{split}
\end{equation}
By Taylor expansion, we have
\begin{equation}
(h(\phi^{n+1}),1)- (h(\phi^{n}),1)=((\frac{\delta H}{\delta \phi})^n, {\phi^{n+1}-\phi^n})+\frac 12(h''(\xi^n)({\phi^{n+1}-\phi^n}),{\phi^{n+1}-\phi^n}).
\end{equation}
We can then conclude from the above two relations.
\begin{prop}
  The scheme \eqref{sav:1b}-\eqref{sav:4b} satisfies the following energy law:
 
\begin{equation}
 \tilde E^{n+1}-\tilde E^n + \frac {\lambda^{n+1}}2(h''(\xi^n)({\phi^{n+1}-\phi^n}),{\phi^{n+1}-\phi^n})\le -\delta t(\mathcal{G}\mu^{n+1},\mu^{n+1}),
\end{equation}
where 
$$\tilde E^n =\frac1{2} (\mathcal{L}\phi^{n},\phi^{n})+(r^n)^2.$$
\end{prop}

\begin{remark}
 Assuming $h''(\phi)\ge 0$ for all $\phi$ and $\lambda^{n+1}\ge 0$ for all $n$,  then the above result indicates that  the scheme \eqref{sav:1b}-\eqref{sav:4b} is unconditionally energy stable.  Note that for some applications, we have $h''(\phi)\ge 0$ and one can show that $\lambda(t)>0$  (cf. \cite{Caf.L09}).  But we are unable to show  $\lambda^{n+1}\ge 0$ for all $n$ under suitable conditions. However,  our numerical results indicate that this is true at least for the examples we considered in this paper.
\end{remark}

\subsection{The second approach}
The main drawback of the first approach is that we can not rigorously prove that the scheme is energy dissipative.  We present below an approach which is efficient  as  the first approach but is energy stable.  The key idea is to  introduce another Lagrange  multiplier $\eta(t)$ to enforce the energy dissipation.  More precisely,  we rewrite \eqref{sav:1}-\eqref{sav:4} as follows:
\begin{eqnarray}
&&\partial_t\phi=-\mathcal{G}\mu,\label{lsav:1}\\
&&\mu=\mathcal{L}\phi+\eta(t) F'(\phi)-\lambda(t) \frac{\delta H}{\delta \phi},\label{lsav:2}\\
&&\frac{d }{d t} \int_\Omega h(\phi)d\bx =0,\label{lsav:3}\\
&&\frac{d }{d t} \int_\Omega F(\phi)d\bx=\eta(t)({F'(\phi)},\phi_t)-\lambda(t) (\frac{\delta H}{\delta \phi}, \phi_t).\label{lsav:4}
\end{eqnarray}
Note that the last term in \eqref{lsav:4} is zero thanks to \eqref{lsav:3}.  We added this zero term here for the sake of constructing energy stable schemes below.

Taking the inner products of the first two equations with $\mu$ and $-\phi_t$ respectively, summing up the results along with the fourth-equation and using the third equation,  we obtain the following energy dissipation law:
\begin{equation}
\frac{d }{d t}  E(\phi)=-(\mathcal{G}\mu,\mu),
\end{equation}
where $ E(\phi)$ is the original energy defined in
\eqref{orienergy}.

For example,  a second-order  scheme based on Crank-Nicolson can be constructed as follows:
\begin{eqnarray}
&&\frac{\phi^{n+1}-\phi^n}{\delta t}=-\mathcal{G}\mu^{n+1/2},\label{lsav:1b}\\
&&\mu^{n+1/2}=\mathcal{L}\phi^{n+1/2}+{\eta^{n+1/2}}F'(\phi^{*,n+1/2})-\lambda^{n+1/2} ({\frac{\delta H}{\delta \phi}})^{*,n+1/2},\label{lsav:2b}\\
&& H(\phi^{n+1})=H(\phi^{0}),\label{lsav:3b}\\
&&\int_\Omega F(\phi^{n+1})- F(\phi^{n})d\bx =\eta^{n+1/2}({F'(\phi^{*,n+1/2})},{\phi^{n+1}-\phi^n})\label{lsav:4b}\\
&&\hskip 1cm -\lambda^{n+1/2} ( ({\frac{\delta H}{\delta \phi}})^{*,n+1/2}, {\phi^{n+1}-\phi^n}),\nonumber
\end{eqnarray}
where $f^{n+1/2}=\frac 12(f^{n+1}+f^n)$ and $f^{*,n+1/2}=\frac 12(3f^n-f^{n-1})$  for any sequence $\{f^n\}$.  Note that unlike in the continuous case,  the last term in \eqref{lsav:4b} is no longer zero,  it is a second-order approximation to zero.  This term is necessary for the unconditional stability that we show below.
 Taking the inner products of \eqref{lsav:1b} with $\mu^{n+1/2}$ and  of 
\eqref{lsav:2b} with $-\frac{\phi^{n+1}-\phi^n}{\delta t}$,  summing up the results along with \eqref{lsav:4b},  we immediately derive the following results:
\begin{thm}
The scheme \eqref{lsav:1b}-\eqref{lsav:4b} is unconditionally energy stable in the sense that
$$E(\phi^{n+1})-E(\phi^n)= -\delta t(\mathcal{G}\mu^{n+1/2},\mu^{n+1/2}),$$
where $E(\phi)$ is  the original energy defined in \eqref{orienergy}.
\end{thm}

The above scheme can be efficiently implemented as the previous two schemes.  Indeed, writing 
\begin{equation}\label{lpmr}
\phi^{n+1}=\phi_1^{n+1}+\eta^{n+1/2}\phi_2^{n+1}+\lambda^{n+1/2}\phi_3^{n+1},\; \mu^{n+1}=\mu_1^{n+1}+\eta^{n+1/2}\mu_2^{n+1}+\lambda^{n+1/2}\mu_3^{n+1},
 \end{equation}
in the scheme \eqref{lsav:1b}-\eqref{lsav:4b},  we find that $(\phi_i^{n+1},\mu_i^{n+1})\;(i=1,2,3)$ can be determined  as follows:
\begin{eqnarray}
&&\frac{\phi_1^{n+1}-\phi^n}{\delta t}=-\mathcal{G}\mu_1^{n+1/2},\label{lsav:1c}\\
&&\mu_1^{n+1/2}=\mathcal{L}\phi_1^{n+1/2};\label{lsav:2c}
\end{eqnarray}

\begin{eqnarray}
&&\frac{\phi_2^{n+1}}{\delta t}=-\mathcal{G}\mu_2^{n+1/2},\label{lsav:1d}\\
&&\mu_2^{n+1/2}=\mathcal{L}\phi_2^{n+1/2}+F'(\phi^{*,n+1/2})\label{lsav:2d};\end{eqnarray}
and
\begin{eqnarray}
&&\frac{\phi_3^{n+1}}{\delta t}=-\mathcal{G}\mu_3^{n+1/2},\label{lsav:1e}\\
&&\mu_3^{n+1/2}=\mathcal{L}\phi_3^{n+1/2}-({\frac{\delta H}{\delta \phi}})^{*,n+1/2}\label{lsav:2e}.\end{eqnarray}
The above three linear systems with constant coefficients can be easily solved.
Once we determine $(\phi_i^{n+1},\mu_i^{n+1})\;(i=1,2,3)$ from the above,  it remains to solve for $(\eta^{n+1/2},\lambda^{n+1/2})$.  To this end, we plug \eqref{lpmr} in \eqref{lsav:3b} and \eqref{lsav:4b},  leading to a coupled nonlinear algebraic system for $(\eta^{n+1/2},\lambda^{n+1/2})$.  The complexity of this nonlinear algebraic equation depends on $F(\phi)$ and $h(\phi)$.
\begin{remark}
The coupled nonlinear algebraic system for $(\eta^{n+1/2},\lambda^{n+1/2})$ can be 
 solved by  Newton iteration. Since the exact solution $\eta(t)\equiv 1$, we can use $1$ as the initial guess for $\eta^{n+1/2}$, and still use the linear scheme  \eqref{lsav:1b}-\eqref{lsav:1d}, or  its second-order version based on Crank-Nicolson, to produce an initial guess for $\lambda^{n+1/2}$. With this set of initial guess,  the Newton iteration for the coupled nonlinear algebraic system would converge quickly if $\Delta t$ is not too large.  
\end{remark}

\subsection{The third approach}
In the second approach, one needs to solve a coupled nonlinear algebraic system for $(\lambda^{n+1/2},\eta^{n+1/2})$. The Newton's iteration may fail to converge if $\delta t$ is not sufficiently small. We propose below a modified version in which one can solve $\lambda^{n+1/2}$ first as in the first approach and then determine $\eta^{n+1/2}$ from a nonlinear algebraic equation:

\begin{eqnarray}
&&\frac{\phi^{n+1}-\phi^n}{\delta t}=-\mathcal{G}\mu^{n+1/2},\label{lsav:1b2}\\
&&\mu^{n+1/2}=\mathcal{L}\phi^{n+1/2}+{\eta^{n+1/2}}F'(\phi^{*,n+1/2})-\lambda^{n+1/2} ({\frac{\delta H}{\delta \phi}})^{*,n+1/2},\label{lsav:2b2}\\
&& H(\bar\phi^{n+1})=H(\bar \phi^{0}),\label{lsav:3b2}\\
&&\int_\Omega F(\phi^{n+1})- F(\phi^{n})d\bx =\eta^{n+1/2}({F'(\phi^{*,n+1/2})},{\phi^{n+1}-\phi^n})\label{lsav:4b2}\\
&&\hskip 1cm-\lambda^{n+1/2} ( ({\frac{\delta H}{\delta \phi}})^{*,n+1/2}, {\phi^{n+1}-\phi^n}),\nonumber
\end{eqnarray}
where  $f^{n+1/2}=\frac 12(f^{n+1}+f^n)$ and $f^{*,n+1/2}= \frac12(3f^n-f^{n-1})$  for any sequence $\{g^n\}$, and $\bar \phi^{n+1}=\phi_1^{n+1}+\phi_2^{n+1} +\lambda^{n+1/2} \phi_3^{n+1}$ with $(\phi_i^{n+1}, \;i=1,2,3)$ being the solutions of 
\eqref{lsav:1c}-\eqref{lsav:2c}, \eqref{lsav:1d}-\eqref{lsav:2d} and \eqref{lsav:1e}-\eqref{lsav:2e} respectively.

\begin{remark}
The only difference between the above scheme and the scheme \eqref{lsav:1b}-\eqref{lsav:4b} is that 
$\phi^{n+1}$ in \eqref{lsav:3b}  is replaced by $\bar\phi^{n+1}$ in \eqref{lsav:3b2} which is independent of  $\eta^{n+1/2}$. This is reasonable since $\eta^{n+1/2}$ is an approximation of 1.  As a consequence,  the global  constraint is satisfied with $\{\bar\phi^k\}$ instead of $\{\phi^k\}$.

\end{remark}

The scheme \eqref{lsav:1b2}-\eqref{lsav:4b2} can be efficiently implemented as follows: 
\begin{itemize}
 \item Write $(\phi^{n+1},\mu^{n+1})$ as in \eqref{lpmr} and solve $(\phi_i^{n+1},\mu_i^{n+1},\;i=1,2,3)$ from \eqref{lsav:1c}-\eqref{lsav:2c}, \eqref{lsav:1d}-\eqref{lsav:2d} and \eqref{lsav:1e}-\eqref{lsav:2e}.

\item  Determine $\lambda^{n+1/2}$ from \eqref{lsav:3b2}. This is a nonlinear algebraic equation for $\lambda^{n+1/2}$,  so it can be solved with  Newton iteration by  using the linear scheme   \eqref{lsav:1b}-\eqref{lsav:1d},  or  its second-order version based on Crank-Nicolson,  to produce an initial guess for $\lambda^{n+1/2}$.
\item With $\lambda^{n+1/2}$  known,  Determine $\eta^{n+1/2}$ explicitly from \eqref{lsav:4b2} which is a nonlinear algebraic equation for $\eta^{n+1/2}$.
\end{itemize}

Exactly as for the scheme \eqref{lsav:1b}-\eqref{lsav:4b},  we have the following result:
\begin{thm}
The scheme \eqref{lsav:1b2}-\eqref{lsav:4b2} is unconditionally energy stable in the sense that
$$E(\phi^{n+1})-E(\phi^n)= -\delta t(\mathcal{G}\mu^{n+1/2},\mu^{n+1/2}),$$
where $E(\phi)$ is  the original energy defined in \eqref{orienergy}.
\end{thm}

\subsection{Stabilization and  adaptive time stepping}
For problems with stiff nonlinear terms, one may have to use very small time steps to obtain accurate results with any of the three approaches above. In order to allow larger time steps while achieving desired accuracy, we may add suitable stabilization and use  adaptive time stepping. 
\subsubsection{Stabilization}
Instead of solving \eqref{grad:flow}, we consider a perturbed  system with two additional stabilization terms
\begin{equation}\label{grad:flow2}
\begin{split}
&\phi_t=-\mathcal{G}\mu,\\
&\mu=\mathcal{L}\phi+\epsilon_1 \phi_{tt}+\epsilon_2 \mathcal{L}\phi_{tt}+F'(\phi)-\lambda \frac{\delta H}{\delta \phi},\\
& \frac{d }{d t} H(\phi) =0,
\end{split}
\end{equation}
where $\epsilon_i,\;i=1,2$ are two small stabilization constants whose choices will depend on how stiff are the nonlinear terms. It is easy to see that the above system is a gradient flow with a perturbed free energy $E_\epsilon(\phi)=E(\phi)+\frac{\epsilon_1}2(\phi_t,\phi_t)+\frac{\epsilon_2}2(\mathcal{L}\phi_t,\phi_t)$ and satisfies the following energy law:
\begin{equation}
 \frac d{dt}E_\epsilon(\phi)=-(\mathcal{G}\mu,\mu).
\end{equation}
The schemes presented before for \eqref{grad:flow} can all be easily extended for  \eqref{grad:flow2} while keeping the same simplicity. For example, a second order scheme based on the second approach is:
\begin{eqnarray}
&&\frac{\phi^{n+1}-\phi^n}{\delta t}=-\mathcal{G}\mu^{n+1/2},\label{lsav:1c2}\\
&&\mu^{n+1/2}=\mathcal{L}\phi^{n+1/2}+ \frac{\epsilon_1}{(\delta t)^2}(\phi^{n+1}-2\phi^n+\phi^{n-1})
\\&&\hskip 1cm +\frac{\epsilon_2}{(\delta t)^2}\mathcal{L}(\phi^{n+1}-2\phi^n+\phi^{n-1}) \label{lsav:2c2}
\\&&\hskip 1cm +{\eta^{n+1/2}}F'(\phi^{*,n+1/2})-\lambda^{n+1/2} ({\frac{\delta H}{\delta \phi}})^{*,n+1/2},\nonumber\\
&& H(\phi^{n+1})=H(\phi^{0}),\label{lsav:3c2}\\
&&\int_\Omega F(\phi^{n+1})- F(\phi^{n})d\bx =\eta^{n+1/2}({F'(\phi^{*,n+1/2})},{\phi^{n+1}-\phi^n})\label{lsav:4c2}\\
&&\hskip 1cm-\lambda^{n+1/2} ( ({\frac{\delta H}{\delta \phi}})^{*,n+1/2}, {\phi^{n+1}-\phi^n}),\nonumber
\end{eqnarray}
where$f^{n+1/2}=\frac 12(f^{n+1}+f^n)$ and $f^{*,n+1/2}=\frac 12(3f^n-f^{n-1})$  for any sequence $\{f^n\}$. 

Taking the inner products of \eqref{lsav:1c2} with $\mu^{n+1/2}$ and  of 
\eqref{lsav:2c2} with $-\frac{\phi^{n+1}-\phi^n}{\delta t}$, summing up the results along with \eqref{lsav:4c2} and dropping some unnecessary terms, we immediately derive the following results:
\begin{thm}
The scheme \eqref{lsav:1c2}-\eqref{lsav:4c2} is unconditionally energy stable in the sense that
$$E_\epsilon^{n+1}-E_\epsilon^n\le -\delta t(\mathcal{G}\mu^{n+1/2},\mu^{n+1/2}),$$
where $E_\epsilon^k=E(\phi^k)+\frac{\epsilon_1}{2}(\frac{\phi^k-\phi^{k-1}}{\delta t},\frac{\phi^k-\phi^{k-1}}{\delta t})+\frac{\epsilon_2}{2}(\mathcal{L}\frac{\phi^k-\phi^{k-1}}{\delta t},\frac{\phi^k-\phi^{k-1}}{\delta t}) $ with $E(\phi)$ being the original free energy defined in \eqref{orienergy}.
\end{thm}
It is clear that the above scheme can be efficiently implemented as the scheme 
\eqref{lsav:1b}-\eqref{lsav:4b}.

\subsection{Adaptive time stepping}
A main advantage of unconditionally stable schemes, such as the schemes using second and third approaches, is that one can choose  time steps solely based on the accuracy requirement. Hence, a suitable adaptive time stepping can greatly improve the efficiency. There are many different strategies for adaptive time stepping, we refer to \cite{SXY19} for some simple strategies which have proven to be effective for the SAV related approaches. 

\section{A single component system with multiple constraints}
The three approaches presented in the last section can be easily extended to gradient flows with multi-components and/or multi global constraints. We consider in this section  a single component system with two global constraints.  
\subsection{The model}
Vesicle membranes are formed by lipid bi-layers which play an essential role in biological functions and its equilibrium shapes
often characterized by bending energy and two physical  constraints as described as below.

 As in \cite{du2008numerical,cheng2018multiple}, we consider the bending energy 
\begin{equation}\label{blend:en}
E_b(\phi)=\frac{\epsilon}{2}\int_{\Omega}\Big(-\Delta\phi+\frac{1}{\epsilon^2}G(\phi)\Big)^2d\bx =\frac{\epsilon}{2}\int_{\Omega}w^2d\bx,
\end{equation}
where
\begin{equation*}
w:=-\Delta\phi+\frac{1}{\epsilon^2}G(\phi),\quad G(\phi):=F'(\phi),\;F(\phi)=\frac 14(\phi^2-1)^2.
\end{equation*}
In the above, the level set $\{\phi(\bx,t)=0\}$ denotes the vesicle  membrane surface,  while $\{\phi(\bx,t)>0\}$ and $\{\phi(\bx,t)<0\}$ represent the inside and outside of the membrane surface respectively, and $\epsilon$ is width of transition layer.

During the evolution, the membranes also preserve total volume and surface area represented by

\begin{equation}\label{def:AH}
A(\phi)=\int_{\Omega}\phi d\bx \quad and \quad H(\phi)=\int_{\Omega} h(\phi) d\bx\;\text{ with }\; h(\phi)=\frac{\epsilon}{2}|\Grad \phi|^2+\frac{1}{\epsilon}F(\phi).
\end{equation}

We now introduce   two Lagrange multipliers, $\gamma(t)$ and $\lambda(t)$, to enforce the volume and surface area conservations. The corresponding gradient flow reads:
\begin{eqnarray}
&&\phi_t =-M\mu , \label{vesicle:1}\\
&&\mu=-\epsilon\Delta w + \frac{1}{\epsilon}G'(\phi)w+\gamma(t)+\lambda(t)\frac{\delta H}{\delta \phi}, \label{vesicle:2}\\
&& w=-\Delta \phi+\frac{1}{\epsilon^2}G(\phi), \label{vesicle:3}\\
&& \frac{d}{dt} A(\phi)=0,\label{vesicle:4}\\
&& \frac{d}{dt} H(\phi)=0.\label{vesicle:5}
\end{eqnarray}
where $M$ is the mobility constant. 
The boundary conditions can be either one of the following two types:
\begin{eqnarray}
&&(i)\mbox{ periodic; or } (ii)\,\,\partial_{\bf n} \ph|_{\partial\Omega}=\partial_{\bf n} w|_{\partial\Omega}=0, \label{ori:pdefhbd2}
\end{eqnarray}
where $\bf n$ is the unit outward normal on the boundary $\partial\Omega$.

\begin{lem}
The system \eqref{vesicle:1}-\eqref{vesicle:5} with \eqref{ori:pdefhbd2} admits the following energy dissipative law:
\begin{eqnarray}\label{LAW:1}
\frac{d}{dt} E_b(\phi)=-M(\mu,\mu).
\end{eqnarray}
\end{lem}
\begin{proof}
Taking the $L^2$ inner products of \eqref{vesicle:1} with $\mu$, and of \eqref{vesicle:2} with $\phi_t$ and of \eqref{vesicle:3} with $w$,   integrating by parts and summing up the results,  noticing that $(1,\phi_t)=\frac d{dt} A(\phi)=0$ and  $(\frac{\delta H}{\delta \phi},\phi_t)=\frac d{dt} H(\phi)=0$, we obtain the energy dissipative law.
\end{proof}

To simplify the presentation, we shall only construct a scheme using the third approach in the last section, since it is simpler than the second approach while maintaining unconditional energy stability.  Obviously, schemes based on other approaches can be constructed similarly.
We rewrite the blending energy as 
\begin{equation}
\begin{split}
E_b(\phi)&=\frac{\eps}{2}\int_{\Omega} |\Delta \phi|^2d\bx+\frac{\eps}{2}\int_{\Omega} \frac{6}{\eps^2}\phi^2|\Grad \phi|^2 + \frac{1}{\eps^4} (G(\phi))^2-\frac{2}{\eps^2}|\Grad \phi|^2d\bx
\\&=\frac{\eps}{2}\int_{\Omega} |\Delta \phi|^2d\bx+\int_{\Omega} Q(\phi)d\bx,
\end{split}
\end{equation}
where  $Q(\phi)=\frac{\eps}{2}\{ \frac{6}{\eps^2}\phi^2|\Grad \phi|^2 + \frac{1}{\eps^4} (G(\phi))^2-\frac{2}{\eps^2}|\Grad \phi|^2\} $.  The key in the second and third approaches is to introduce a Lagrange multiplier $\eta(t)$ to deal with nonlinear part of the energy $Q(\phi)$ and reformulate  \eqref{vesicle:1}-\eqref{vesicle:5}  as 
\begin{eqnarray}
&&\phi_t =-M\mu , \label{new:vesicle:1}\\
&&\mu=\epsilon\Delta^2 \phi + \eta(t) \frac{\delta Q}{\delta \phi}+\gamma(t)+\lambda(t)\frac{\delta H}{\delta \phi}, \label{new:vesicle:2}\\
&&\frac{d}{dt}\int_{\Omega} Q(\phi)d\bx=\eta(t)(\frac{\delta Q}{\delta \phi},\phi_t)+\lambda(t)(\frac{\delta H}{\delta \phi},\phi_t), \label{new:vesicle:3}\\
&& \frac{d}{dt}A(\phi)=0,\label{new:vesicle:4}\\
&& \frac{d}{dt}H(\phi)=0.\label{new:vesicle:5}
\end{eqnarray}
Note that the last term in \eqref{new:vesicle:3} is zero. We added this term which is essential in constructing efficient energy stable schemes.

\subsection{A second-order scheme based on the third approach} As an example, we construct below a second-order (BDF2) scheme for system \eqref{new:vesicle:1}-\eqref{new:vesicle:5} based on the third approach:  
\begin{eqnarray}
&&\frac{3\phi^{n+1}-4\phi^n+\phi^{n-1}}{2\delta t} =-M\mu^{n+1} , \label{vesicle:sch:1}\\
&&\mu^{n+1}=\epsilon\Delta^2 \phi^{n+1} + \eta^{n+1}({\frac{\delta Q}{\delta \phi}})^{*,n+1} +\gamma^{n+1}+\lambda^{n+1}({\frac{\delta H}{\delta \phi}})^{*,n+1}, \label{vesicle:sch:2}\\
&&\int_{\Omega} 3Q(\phi^{n+1})-4Q(\phi^n)+Q(\phi^{n-1}) d\bx \\&&\hskip 1cm=\eta^{n+1}(({\frac{\delta Q}{\delta \phi}})^{*,n+1},3\phi^{n+1}-4\phi^n+\phi^{n-1})\label{vesicle:sch:3}\\
&&\hskip 1cm+\lambda^{n+1}(({\frac{\delta H}{\delta \phi}})^{*,n+1},3\phi^{n+1}-4\phi^n+\phi^{n-1}), \nonumber\\
&& \int_{\Omega}\bar\phi^{n+1} d\bx=\int_{\Omega}\phi^{0} d\bx,\label{vesicle:sch:4}\\
&& H(\bar\phi^{n+1})=H(\phi^{0}),\label{vesicle:sch:5}
\end{eqnarray}
where $ g^{*,n+1}=2g^n-g^{n-1}$ for any sequence $\{g^n\}$, $\bar\phi^{n+1}$ is defined in \eqref{eq:hat}  below during the solution procedure.

Setting
\begin{equation}\label{phi4}
\phi^{n+1}=\phi_1^{n+1}+\eta^{n+1}\phi_2^{n+1}+\gamma^{n+1}\phi_3^{n+1}+\lambda^{n+1}\phi_4^{n+1},
\end{equation}
in \eqref{vesicle:sch:1}-\eqref{vesicle:sch:2} and eliminating $\mu^{n+1}$, we find that $\{\phi^{n+1}_i\}$ can be determined by
\begin{equation}
(\frac{1}{2\delta t}+M\eps\Delta^2)\phi^{n+1}_i=g_i, \quad i=1,2,3,4, 
\end{equation}
with $g_i$ to be known functions from previous steps. Once $\{\phi^{n+1}_i, \;i=1,2,3,4\}$ are known, we define
\begin{eqnarray}
&&\bar \phi^{n+1}=\phi_1^{n+1}+\phi_2^{n+1}+\gamma^{n+1}\phi_3^{n+1}+\lambda^{n+1}\phi_4^{n+1}.\label{eq:hat}
\end{eqnarray}
Note that $\bar \phi^{n+1}$ is still as good an approximation to $\phi|_{t^{n+1}}$  as $\phi^{n+1}$ since $\eta^{n+1}$ is a second-order approximation to 1. 
 
 We can then determine the three Lagrange multipliers  as follows:
 \begin{itemize}
\item Plug \eqref{eq:hat} into \eqref{vesicle:sch:4}, we obtain a linear relation between  $\gamma^{n+1}$ and $\lambda^{n+1}$;
\item Plug \eqref{eq:hat} into \eqref{vesicle:sch:5} and using the linear relation between  $\gamma^{n+1}$ and $\lambda^{n+1}$, we obtain  a  nonlinear algebraic equation for $\lambda^{n+1}$ which can be solved a Newton iteration using an initial guess obtained by  a linear scheme as in Section 2.1;
\item With  $\gamma^{n+1}$ and $\lambda^{n+1}$ known,  determine $\eta^{n+1}$ by  plugging \eqref{phi4}  into \eqref{vesicle:sch:3}  and solve the resulted   nonlinear algebraic equation with the initial guess $1$.
\end{itemize}
Hence, the above scheme can be  implemented very efficiently.
As for the stability, we have the following result:
\begin{thm}
 The  scheme \eqref{vesicle:sch:1}-\eqref{vesicle:sch:4} is unconditionally energy stable in the sense that
 \begin{equation*}
 E_b^{n+1}-E_b^{n}\le -\delta tM \|\mu^{n+1}\|^2,
\end{equation*}
where 
\begin{equation}
E_b^{n+1}=\frac{\eps}{4}(\|\Delta\phi^{n+1}\|^2+\|\Delta(2\phi^{n+1}-\phi^n)\|^2)+\frac 12\int_{\Omega} 3Q(\phi^{n+1})-Q(\phi^n)) d\bx,
\end{equation}
which is a second-order approximation to the original free energy $E_b(\phi)$ at $t^{n+1}$.
\end{thm}
\begin{proof}
Taking the inner product of equation \eqref{vesicle:sch:1} with $2\delta t\mu^{n+1}$, we derive
\begin{equation}
(3\phi^{n+1}-4\phi^n+\phi^{n-1},\mu^{n+1})=-2\delta tM\|\mu^{n+1}\|^2.
\end{equation}
Due to equation \eqref{vesicle:sch:4}, we have
\begin{equation}\label{ves:eq}
(1,3\phi^{n+1}-4\phi^n+\phi^{n-1})=0.
\end{equation}
Taking the inner product of equation \eqref{vesicle:sch:2} with $3\phi^{n+1}-4\phi^n+\phi^{n-1}$  and using equality  \eqref{ves:eq} and equation \eqref{vesicle:sch:3}, we derive
\begin{equation}
\begin{split}
&(3\phi^{n+1}-4\phi^n+\phi^{n-1},\mu^{n+1})=(\epsilon\Delta^2 \phi^{n+1} ,3\phi^{n+1}-4\phi^n+\phi^{n-1})\\&+\int_{\Omega} 3Q(\phi^{n+1})-4Q(\phi^n)+Q(\phi^{n-1}) d\bx.
\end{split}
\end{equation}
Using the identity
\begin{equation}\label{bdf2}
\begin{split}
 2(a^{n+1}, 3a^{n+1}-4a^n+a^{n-1}) =& \|a^{n+1}\|^2-\|a^{n}\|^2+\|a^{n+1}-2a^n+a^{n-1}\|^2\\
 &+\|2a^{n+1}-a^n\|^2-\|2a^n-a^{n-1}\|^2,
 \end{split}
\end{equation}
 we have
\begin{equation*}
\begin{split}
(\epsilon\Delta^2 \phi^{n+1} ,3\phi^{n+1}-4\phi^n+\phi^{n-1})&=\frac{\eps}{2}(\|\Delta\phi^{n+1}\|^2-\|\Delta\phi^n\|^2+\|\Delta(2\phi^{n+1}-\phi^n)\|^2\\
&-\|\Delta(2\phi^{n}-\phi^{n-1})\|^2+\|\Delta(\phi^{n+1}-2\phi^n+\phi^{n-1})\|^2).
\end{split}
\end{equation*}
Combining the above equalities and dropping some unnecessary terms, we arrive at the desired result.
\end{proof}

\section{A multi-component system with multiple constraints} 

We consider in this section a norm-preserving model for optimal partition written in the form of gradient flow. It is a  multi-component system with multiple constraints.
\subsection{The model}
The optimal partition problem can be described by a norm-preserving gradient dynamics \cite{CL1}. Given a positive integer $m$ and a small parameter $\epsilon$, the total free energy is given by \begin{equation}
E(\bphi)=\int_{\Omega}(\frac 12|\Grad \bphi|^2+\bm{F}(\bphi))d\bx,
\end{equation}
where   $\bphi\in [H_0^1(\Omega)]^m$ is a vector valued function satisfying the norm constraint
\begin{equation}\label{norm-1}
H_j(\phi):=\int_{\Omega}|\phi_j|^2d\bx=1,\quad j=1,2, \dots ,m,
\end{equation} 
  $\bm{F}$ represents interaction potential of each partition
\begin{equation}
\bm{F}(\bphi)=\frac{1}{\eps^2}\sum\limits_{i=1}^{m}\sum\limits_{j<i}\phi_i^2\phi_j^2.
\end{equation}

We shall enforce the normalization conditions \eqref{norm-1} by introducing $j$ Lagrange multipliers. The corresponding gradient flow reads
\begin{eqnarray}
&&\partial_t\phi_j=-\mu_j,\label{part:1}\\
&&\mu_j=-\Delta \phi_j-\lambda_j(t)\phi_j+\frac{\delta \bm{F}}{\delta \phi_j},\label{part:2}\\
&&\frac{d}{dt}\int_{\Omega}|\phi_j(x,t)|^2d\bx=0,\quad j=1,2,\dots ,m, \label{part:3}
\end{eqnarray}
with initial condition $\int_{\Omega}|\phi_j(x,0)|^2d\bx=1$,
 boundary conditions 
\begin{eqnarray}
&&(i)\mbox{ periodic; or } (ii)\,\,\bphi|_{\partial\Omega}=0. \label{ori:part:bd}
\end{eqnarray}
\begin{lem}
The system \eqref{part:1}-\eqref{part:3} with \eqref{ori:part:bd} admits the following energy dissipative law:
\begin{eqnarray}\label{LAW:0}
\frac{d}{dt} E(\bphi)=-M\int_\Omega \sum\limits_{j=1}^{m}\mu_j^2 d\bx.
\end{eqnarray}
\end{lem}
\begin{proof}
Taking the  inner products of \eqref{part:1} with $\mu_j$, and of \eqref{part:2} with $\partial_t\phi_j,j=1,2,\dots ,m$,  noticing the equality \eqref{part:3}, integrating by parts and summing up all the relations, we obtain the desired result.
\end{proof}
Again the key for the second  and third approaches is to introduce  a Lagrange multiplier to deal with the nonlinear term, and rewrite
the system \eqref{part:1}-\eqref{part:3}  as 
\begin{eqnarray}
&&\partial_t\phi_j=-\mu_j,\label{new:part:1}\\
&&\mu_j=-\Delta \phi_j-\lambda_j(t)\phi_j+\eta(t)\frac{\delta \bm{F}}{\delta \phi_j},\label{new:part:2}\\
&&\frac{d}{dt}\int_{\Omega}\bm{F}(\bphi)d\bx=\sum\limits_{j=1}^{m}\eta(t)(\frac{\delta \bm{F}}{\delta \phi_j},\partial_t\phi_j)+\sum\limits_{j=1}^{m}\lambda_j(t)(\phi_j,\partial_t\phi_j),\label{new:part:3}\\
&&\frac{d}{dt}\int_{\Omega}|\phi_j(x,t)|^2d\bx=0,\quad j=1,\dots m.\label{new:part:4}
\end{eqnarray}
Note that we added the last term in \eqref{new:part:3} which  is zero  but is essential in constructing energy stable schemes below.

\subsection{A second-order  scheme based on the third approach}

As an example, we construct below a second-order (BDF2) scheme for the system  \eqref{part:1}-\eqref{part:3} based on the third approach: 

We can construct a second-order scheme based on system \eqref{new:part:1}-\eqref{new:part:3}.

For $j=1,2,\cdots,m$:
\begin{eqnarray}
&&\frac{3\phi_j^{n+1}-4\phi_j^n+\phi_j^{n-1}}{2\delta t} =-\mu_j^{n+1},\label{np-e-bdf2:1}\\
&&\mu^{n+1}_j=-\Delta \phi_j^{n+1}-\lambda_j^{n+1}\phi_j^{\star,n+1}+\eta^{n+1}f(\phi_j^{\star,n+1}),\label{np-e-bdf2:2}\\
&&\int_{\Omega}3\bm{F}(\bphi^{n+1})-4\bm{F}(\bphi^{n})+\bm{F}(\bphi^{n-1})d\bx\\&&\hskip 1cm=\sum\limits_{j=1}^{m}\{(\eta^{n+1}(\frac{\delta \bm{F}}{\delta \phi_j})^{\star,n+1},3\phi^{n+1}_j-4\phi^n_j+\phi^{n-1}_j)
\label{np-e-bdf2:3}\\
&&\hskip 1cm +(\lambda_j^{n+1}\phi_j^{\star,n+1},3\phi_j^{n+1}-4\phi_j^n+\phi_j^{n-1})\},\nonumber\\
&&\int_{\Omega}|\bar \phi_j^{n+1}|^2d\bx=\int_{\Omega}|\phi_j^{0}|^2d\bx,\label{np-e-bdf2:4}
\end{eqnarray}
where $g^{\star,n+1}=2g^n-g^{n-1}$ for any sequence $\{g^n\}$, and $\bar \phi_j^{n+1}$ is defined in \eqref{eq:barj} below during the solution procedure.

Setting
 \begin{equation}\label{sol:pj}
 \phi_j^{n+1}=\psi^{n+1}_{0,j}+\lambda_j^{n+1}\psi^{n+1}_{1,j}+\eta^{n+1}\psi_{2,j}^{n+1},\quad j=1,2,\cdots ,m,
 \end{equation}
 and plugging the above into \eqref{np-e-bdf2:1}-\eqref{np-e-bdf2:2}, we can
 determine $\psi^{n+1}_{0,j}$, $\psi^{n+1}_{1,j}$ and $\psi^{n+1}_{2,j}$ by solving decoupled linear equations
  \begin{equation}
 (\frac{3}{2\delta t}-\Delta )\psi^{n+1}_{k,j}=g_{k,j},\quad k=0,1,2,\; j=1,2,\cdots ,m,
 \end{equation}
 where $\{g_{k,j}\}$ are known functions from the previous steps.
 Then we define
 \begin{equation}\label{eq:barj}
 \bar \phi_j^{n+1}=\psi^{n+1}_{0,j}+\lambda_j^{n+1}\psi^{n+1}_{1,j}+\psi_{2,j}^{n+1},\quad j=1,2,\cdots ,m.
\end{equation}
Note that $\bar \phi_j^{n+1}$ is still as good an approximation to $\phi_j|_{t^{n+1}}$  as $\phi_j^{n+1}$ since $\eta^{n+1}$ is a second-order approximation to 1.

Finally, we determine $\{\lambda_j^{n+1}\}$ and $\eta^{n+1}$ as follows:
\begin{itemize}
\item  Plug \eqref{eq:barj} into \eqref{np-e-bdf2:4}, we obtain, for each $j$, a quadratic algebraic equation for $\lambda^{n+1}_j$ which can be directly solved;
\item With $\{\lambda_j^{n+1}\}$ known, we plug \eqref{sol:pj} into \eqref{np-e-bdf2:3} to obtain a nonlinear algebraic equation for $\eta^{n+1}$, and we solve the nonlinear algebraic equation by a   Newton iteration with  1 as initial condition.
\end{itemize}
 Hence, the above scheme can be efficiently implemented.
 As for the stability, we have the following result:
\begin{thm}
 The  scheme \eqref{np-e-bdf2:1}-\eqref{np-e-bdf2:4} is unconditionally energy stable in the sense that
 \begin{equation*}
 E^{n+1}-E^{n}\le -\delta t\sum\limits_{j=1}^{m} \|\mu_j^{n+1}\|^2,
\end{equation*}
where 
\begin{equation}\label{vesicle:engy}
E^{n+1}=\frac{1}{4}(\|\Grad\bphi^{n+1}\|^2+\|\Grad(2\bphi^{n+1}-\bphi^n)\|^2)+\frac 12\int_{\Omega} 3\bm{F}(\bphi^{n+1})-\bm{F}(\bphi^{n})d\bx,
\end{equation}
which is a second-order approximation to the original free energy $E(\bphi)$ at $t^{n+1}$.
\end{thm}
\begin{proof}
 Taking the inner product of \eqref{np-e-bdf2:1} with $2\delta t \mu_j^{n+1}$, we derive
 \begin{equation}
 (3\phi_j^{n+1}-4\phi_j^n+\phi_j^{n-1},\mu_j^{n+1})=-2\delta t\|\mu_j^{n+1}\|^2.
 \end{equation}
 Taking the  inner product of \eqref{np-e-bdf2:2} with $3\phi_j^{n+1}-4\phi_j^n+\phi_j^{n-1}$,  and  summing up  all these  equations from $j=1,2,\cdots ,m$, we obtain
 \begin{equation}
 \begin{split}
 \sum\limits_{j=1}^{m}(3\phi_j^{n+1}-4\phi_j^n+\phi_j^{n-1},\mu_j^{n+1})&=\sum\limits_{j=1}^{m}(\nabla \phi_j^{n+1},\nabla(3\phi_j^{n+1}-4\phi_j^n+\phi_j^{n-1}))\\&+\int_{\Omega}3\bm{F}(\bphi^{n+1})-4\bm{F}(\bphi^{n})+\bm{F}(\bphi^{n-1})d\bx.
 \end{split}
 \end{equation}
 Combing all relations obtained  above and  using the identity \eqref{bdf2},  we obtain the desired result.
\end{proof}

\section{Numerical results}
We present in this section some numerical experiments to compare the performance of different approaches and to validate their stability and convergence rates.  In all numerical examples below, we assume  periodic boundary conditions and use a Fourier Spectral  method in space. The computational domain is $[-\pi,\pi)^d$ with $d=2,3$.

\subsection{Validation and comparison}
We consider the   phase field vesicle membrane model \eqref{vesicle:1}-\eqref{vesicle:5} with $ \epsilon=\frac{6\pi}{128}, \; M=1$, and use $128$ modes in each direction in our Fourier Spectral method so that the spatial discretization errors are negligible compared with time discretization error.

\subsubsection{\bf Comparison of the three  approaches}
We first investigate the performance of the three  approaches proposed in Section 2.  We consider the 2D  phase field vesicle membrane model \eqref{vesicle:1}-\eqref{vesicle:5}, and  choose  as initial condition two close-by circles given by
\begin{equation}\label{ini_two_2D}
\phi(x,y,0)=\sum\limits_{i=1}^2\tanh(\frac{r_i-\sqrt{(x-x_i)^2+(y-y_i)^2}}{\sqrt{2}\eps}) + 1. 
\end{equation}
We define $(r_1,r_2)=(0.28\pi,0.28\pi)$, $(x_1,x_2)=(0,0)$ and $(y_1,y_2)=(0.35\pi,-0.35\pi)$.
In the left of  Fig.\,\ref{constraints:compare}, we plot the evolution of  Lagrange multiplier $\lambda$ with respect to time by using BDF2 scheme of  three numerical approaches.  
We observe that the three approaches lead to indistinguishable $\lambda$. However, we have  to use a very small time step, $\delta t=10^{-5}$, in the second approach for the Newton iteration to converge, while larger time steps can be used for the first and third approaches. 
On the other hand, we plot in the right of Fig.\,\ref{constraints:compare} the evolution of the  surface area by using the three approaches. We  observe that the first and second approaches   preserve exactly the surface area,  while very small differences on $B({\phi})$ are observed   by the third approach at several initial time steps, since the third approach only preserves  $B(\bar{\phi})$ instead of $B({\phi})$.

 The above results indicate that the first and third approaches  are preferable over the second approach, since they allow  larger time steps.  Therefore, we shall only use  the first and  third approaches  in the remaining simulations.

\begin{figure}[htbp]
\centering
\includegraphics[width=0.45\textwidth,clip==]{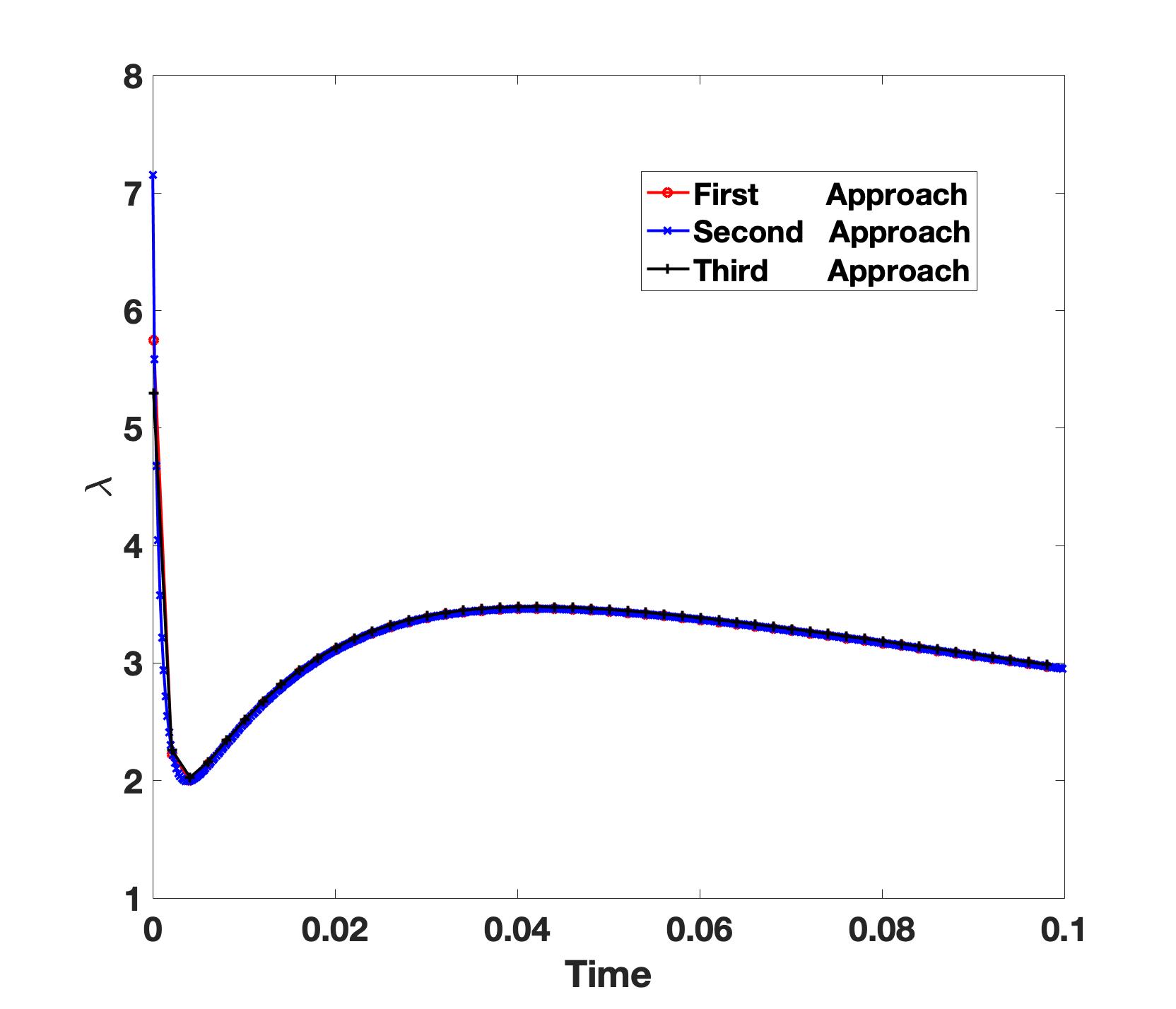}
\includegraphics[width=0.45\textwidth,clip==]{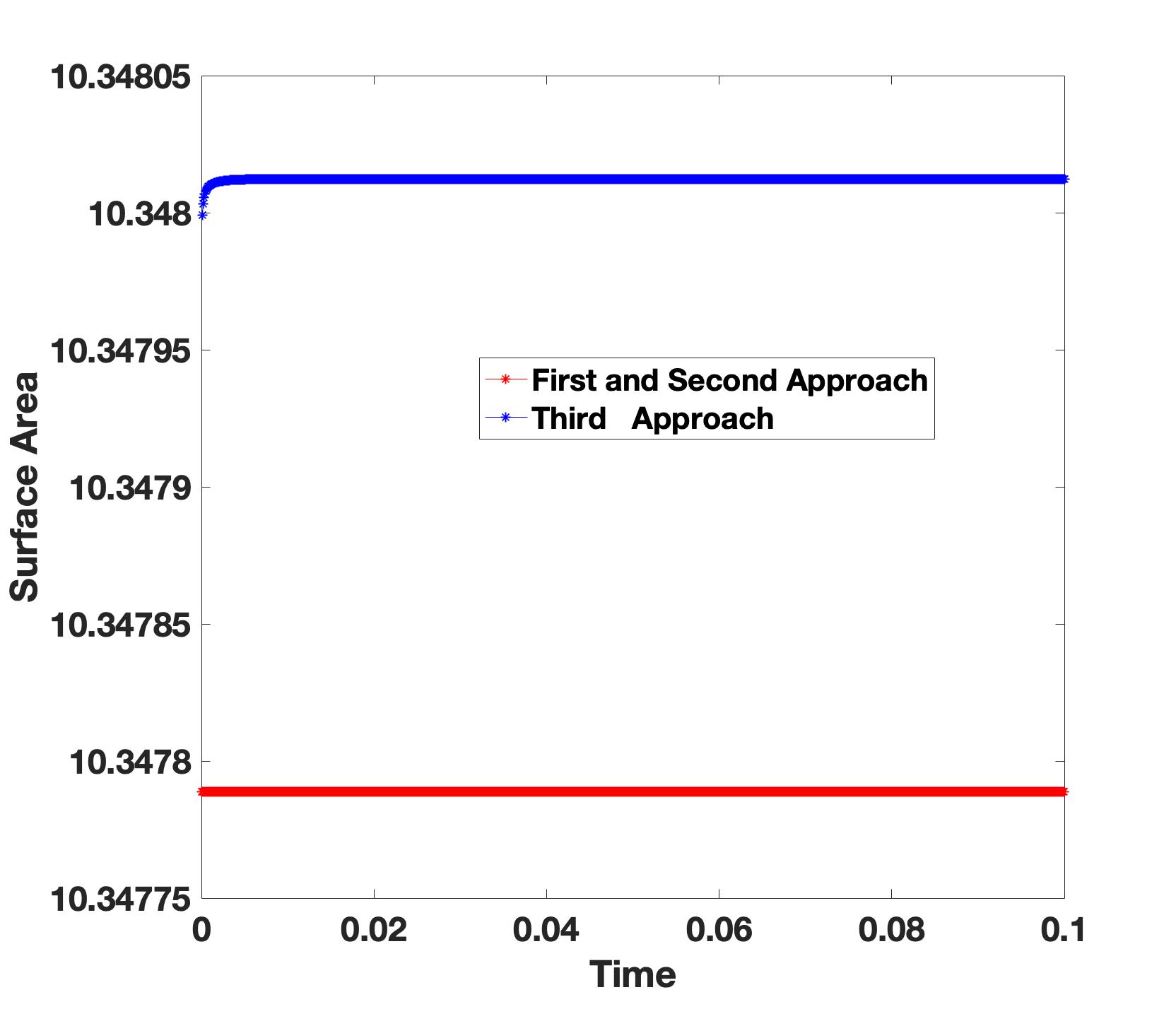}
\caption{Left: Evolution of $\lambda$ by using  three  approaches with $\delta t=10^{-4}$ for the first and third approaches while $\delta t=10^{-5}$ for the second approach; Right: Evolution of the  surface area.}\label{constraints:compare}
\end{figure}

\subsubsection{\bf Convergence rate}
We  test the  convergence rate of BDF2 schemes using first and third approaches for 2D phase field vesicle membrane model \eqref{vesicle:1}-\eqref{vesicle:5} with the initial condition
\begin{equation}
\phi(x,y,0)=(\frac{\sin(2x)\cos(2y)}{4}+0.48)(1-\frac{\sin^2(t)}{2}). 
\end{equation} 
 The  reference solutions are obtained with a small time step $\delta t=10^{-5}$ using the BDF2 schemes.  In Fig,\,\ref{Order_test},  we plot the $L^{\infty}$ errors  of $\phi$ between numerical solution and reference solution  with different time steps. We observe that second-order convergence rates are achieved by both approaches.

\begin{figure}[htbp]
\centering
\subfigure{
\includegraphics[width=0.45\textwidth,clip==]{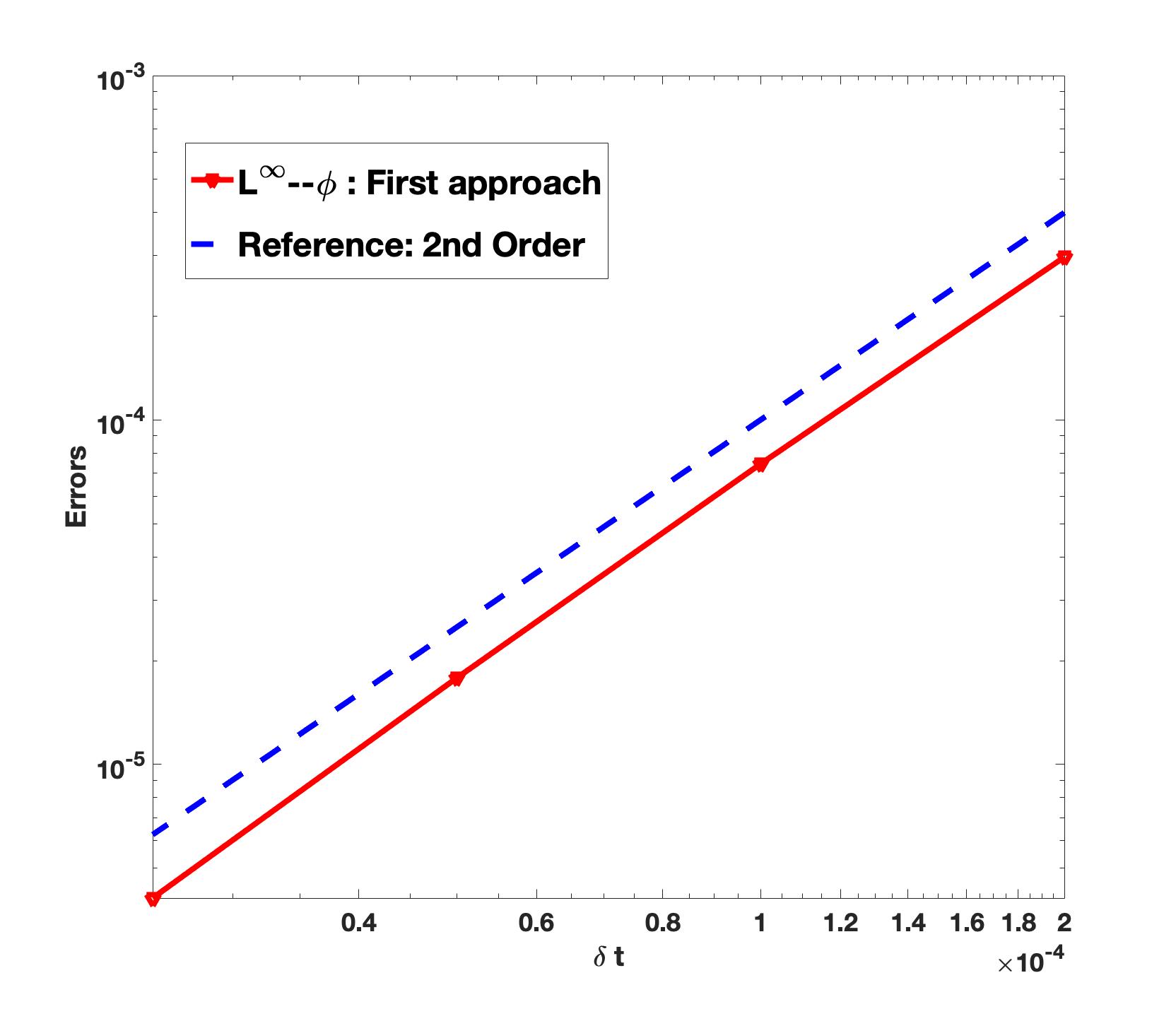}}
\subfigure{
\includegraphics[width=0.45\textwidth,clip==]{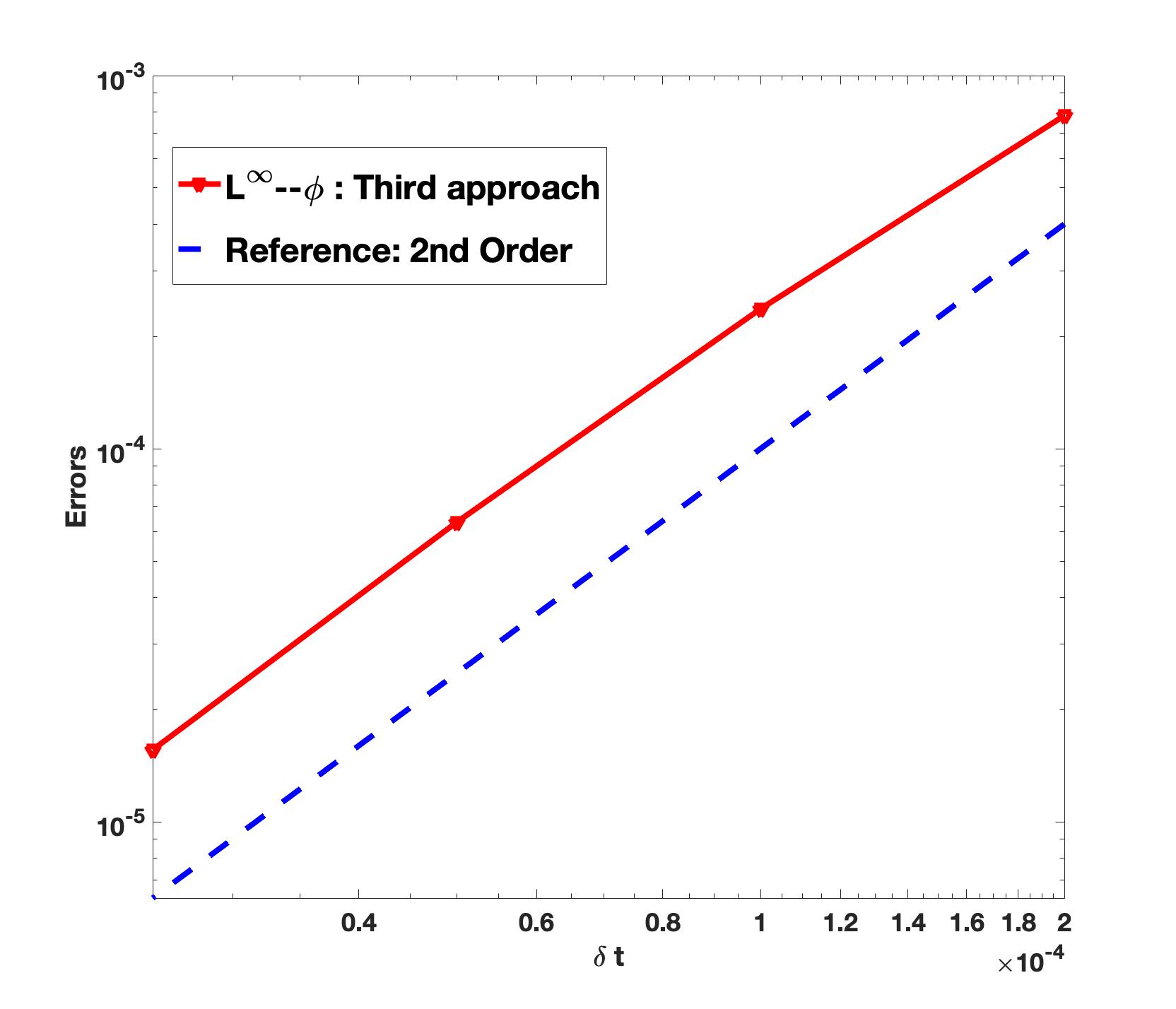}}
\caption{Convergence rate of BDF2 schemes by using the first and third approaches for 2D phase field vesicle membrane model \eqref{vesicle:1}-\eqref{vesicle:5}.}\label{Order_test}
\end{figure}

\subsubsection{\bf Comparison between the new approaches and the penalty approach in \cite{cheng2018multiple}}
We now compare our  Lagrange multiplier approach with the penalty approach developed  in \cite{cheng2018multiple}. 
In the penalty approach, we introduce two penalty parameters $\gamma$ and $\eta$, and consider the total free energy
\begin{equation}
 E_{total}(\phi)=E_b(\phi)+\frac1{2\gamma}(A(\phi)-A(\phi|_{t=0}))^2+\frac1{2\eta}(H(\phi)-H(\phi|_{t=0}))^2,
\end{equation}
where $E_b(\phi)$, $A(\phi)$ and $H(\phi)$ are defined in \eqref{blend:en} and \eqref{def:AH}. 
We observe that the penalty approach can only approximately preserve the constraints on $A(\phi)$ and $H(\phi)$,  and very small penalty parameters  have to be used if we want to preserve the constraints to a high accuracy. However,  small penalty parameters will lead to stiff  systems such that the MSAV approach proposed in \cite{cheng2018multiple}  requires very small time steps to get accurate solutions. More precisely,  we list the  maximum allowable time step in Table \ref{table2} for the MSAV scheme for 2D phase field vesicle membrane model by using penalty approach.  We observe that the  maximum allowable time step behaves like $\min(\sqrt{\gamma},\sqrt{\eta})$. On the other hand, the new Lagrangian multiplier approach is more efficient than the MSAV approach at each time step, and allows much large time steps.

\begin{table}[ht!]
\centering
\begin{tabular}{|r|c|c|c}
\hline
$\delta t$ allowed     & {\bf $\gamma$}  & {\bf  $\eta$}   \\\hline
$2\times 10^{-4}$    &$10^{-5}$  & $10^{-5}$   \\\hline
$2\times 10^{-4}$      &$10^{-6}$  & $10^{-6}$     \\\hline
$1\times 10^{-4}$      &$10^{-7}$  & $10^{-7}$  \\\hline
$5\times 10^{-5}$   &$10^{-8}$  & $10^{-8}$  \\ \hline
$2\times 10^{-5}$   &$10^{-9}$  & $10^{-9}$  \\\hline
$1\times 10^{-5}$   &$10^{-10}$  & $10^{-10}$\\\hline
$2\times 10^{-6}$  &$10^{-11}$  & $10^{-11}$  \\\hline
$1\times 10^{-6}$  &$10^{-12}$  & $10^{-12}$ \\
\hline
\end{tabular}
\vskip 0.5cm
\caption{Largest time step allowed for MSAV scheme with various Penalty parameters $\gamma$ and $\eta$}\label{table2}
\end{table}

Next, we simulate the 3D  phase field vesicle membrane model with the first approach proposed in this paper and the MSAV approach in \cite{cheng2018multiple}. We take  the initial condition as
\begin{equation}\label{ini_four}
\phi(x,y,z,0)=\sum\limits_{i=1}^4\tanh(\frac{r_i-\sqrt{(x-x_i)^2+(y-y_i)^2+(z-z_i)^2}}{\sqrt{2}\eps}) + 3, 
\end{equation}
where $r_i=\frac{\pi}{6}$, $x_i=0$, $(y_1,y_2,y_3,y_4)=(\frac{\pi}{4},-\frac{\pi}{4},\frac{3\pi}{4},-\frac{3\pi}{4})$ and $z_i=0$ for $i=1,2,3,4$.

In Fig.\,\ref{LGM-ball4-3D},  we plot evolution of  the volume difference and surface area difference  by the MSAV scheme in \cite{cheng2018multiple}  with  penalty parameter $\gamma=\eta=10^{-3}$ and by the  BDF2 scheme of  first approach using $\delta t=2\times 10^{-4}$.  We observe  that both  the volume and surface area are  preserved exactly by  the BDF2 scheme of  first approach  while only approximately for the MSAV scheme using the penalty approach.

In Fig.\,\ref{collision3-3D},  we   present snapshots of  isosurface of $\{\phi=0\}$ at different times by using the BDF2 scheme of first approach.  It is observed that the final steady state  is the same  as  that reported   in \cite{cheng2018multiple} using the penalty approach.  We also plot  in Fig.\,\ref{energy}  energy curves of different approaches which  are indistinguishable in all cases. 

\begin{figure}[htbp]
\centering
\subfigure{
\includegraphics[width=0.40\textwidth,clip==]{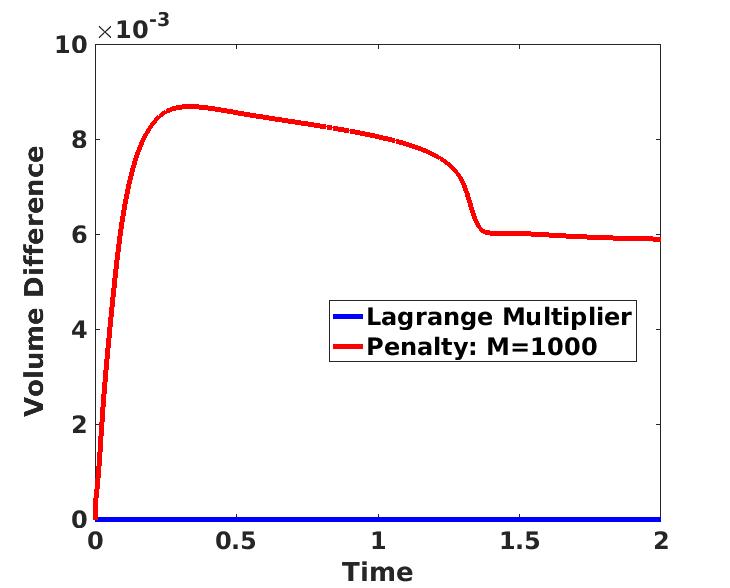}}
\subfigure{
\includegraphics[width=0.40\textwidth,clip==]{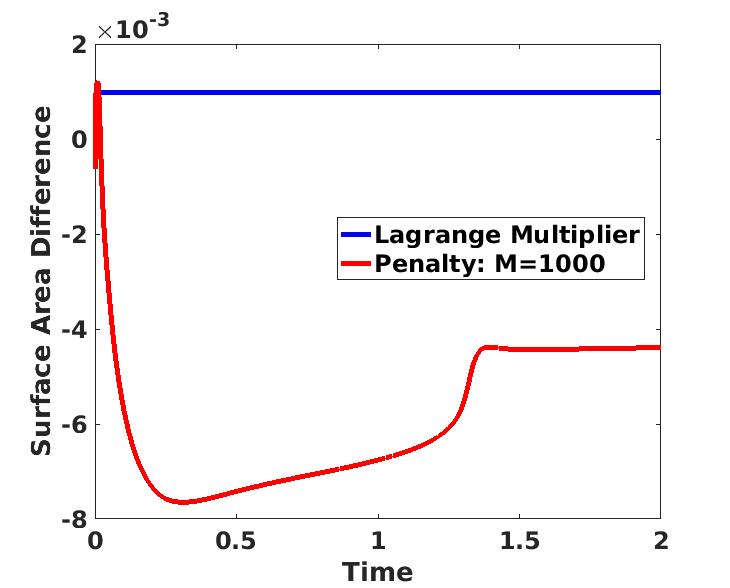}}
\caption{Evolution of the volume  and surface area differences  by the MSAV scheme  $\gamma=\eta=10^{-3}$ and  the BDF2 scheme  of  first approach  with $\delta t=2\times 10^{-4}$.}\label{LGM-ball4-3D}
\end{figure}

\begin{figure}[htbp]
\centering
\includegraphics[width=0.32\textwidth,clip==]{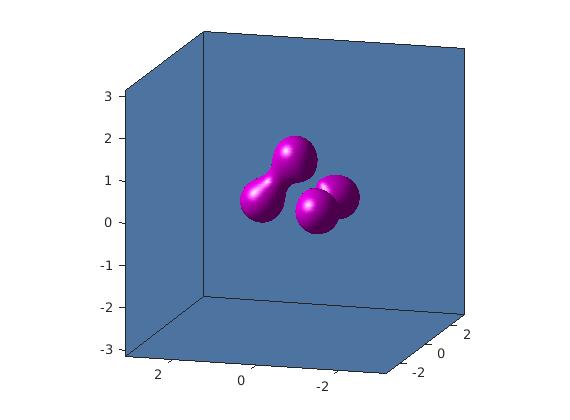}
\includegraphics[width=0.32\textwidth,clip==]{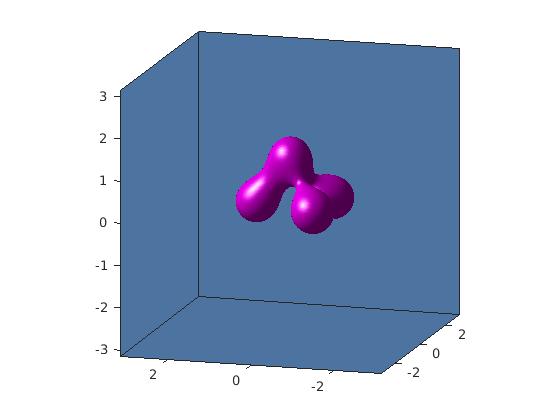}
\includegraphics[width=0.32\textwidth,clip==]{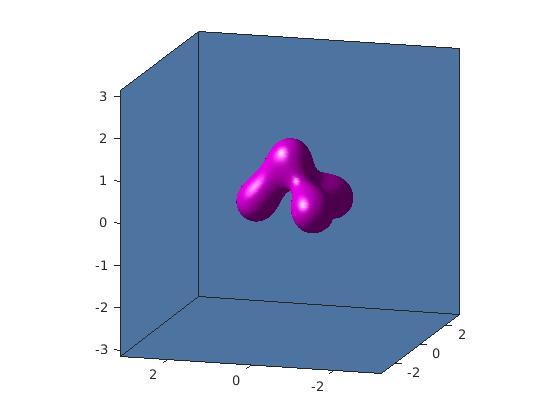}
\includegraphics[width=0.32\textwidth,clip==]{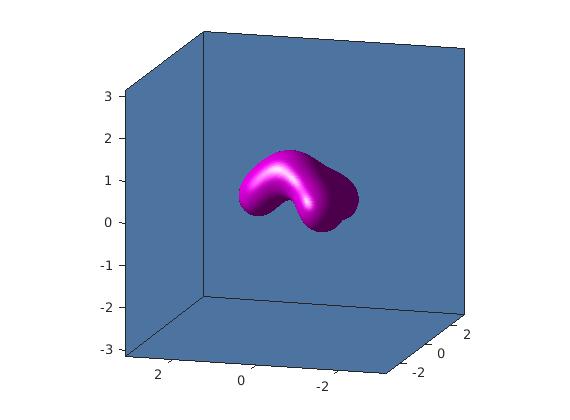}
\includegraphics[width=0.32\textwidth,clip==]{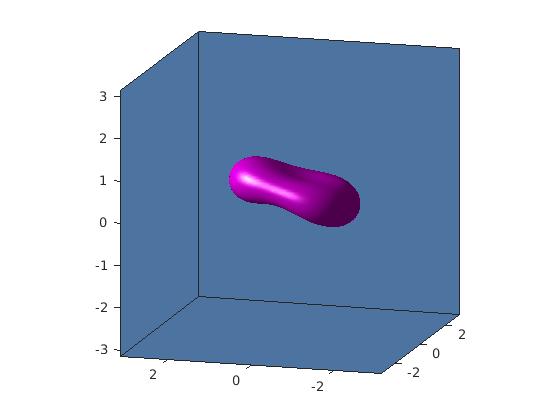}
\includegraphics[width=0.32\textwidth,clip==]{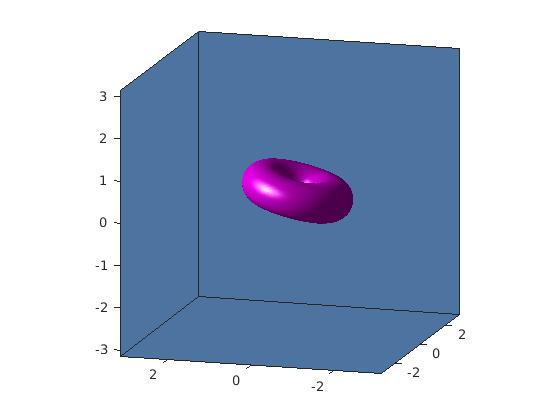}
\caption{Collision of four 3D close-by spherical vesicles by using the BDF2 of first approach) with the time step size $\delta t=2\times 10^{-4}$. Snapshots of  isosurface
of $\phi$  at $t=0.01,0.02, 0.04, 0.2,1,2$.}\label{collision3-3D}
\end{figure}

\begin{figure}[htbp]
\centering
\includegraphics[width=0.45\textwidth,clip==]{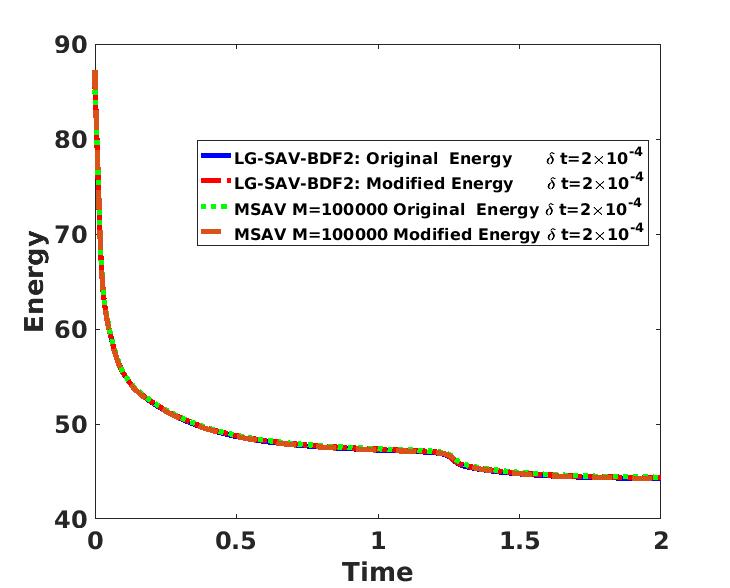}
\caption{Energy evolutions with  different approaches using the initial condition \eqref{ini_four}. }\label{energy}
\end{figure}



\subsection{Additional simulations of 3D vesicle membrane model}
In order to further demonstrate the accuracy and robustness of our new Lagrangian multiplier approach,  we perform  some additional simulations of 3D vesicle membrane model.  As the  first example, we set  four close-by  spheres  as the initial profile  which is formulated as 
\begin{equation}
\phi(x,y,z,0)=\sum\limits_{i=1}^4\tanh(\frac{r_i-\sqrt{(x-x_i)^2+(y-y_i)^2+(z-z_i)^2}}{\sqrt{2}\eps}) + 3, 
\end{equation}
where $r_i=\frac{\pi}{6}$, $x_i=0$, $(y_1,y_2,y_3,y_4)=(\frac{\pi}{4},-\frac{\pi}{4},\frac{3\pi}{4},-\frac{3\pi}{4})$ and $z_i=0$ for $i=1,2,3,4$.
In Fig.\,\ref{cylinder4-3D}, we plot several snapshots  of the iso-surface  $\{\phi=0\}$ by using BDF2  scheme of third  approach with  $\delta t=10^{-4}$.   It is observed that initially separated four spheres connect with each other at  $t=0.02$ and gradually merge into a  cylinder shape at $t=1$.  This  is consistent with results in \cite{du2006simulating}.

As the  second example,  we start with a more complicated initial condition given by
\begin{equation}\label{ini_six}
\phi(x,y,z,0)=\sum\limits_{i=1}^6\tanh(\frac{r_i-\sqrt{(x-x_i)^2+(y-y_i)^2+(z-z_i)^2}}{\sqrt{2}\eps}) + 5, 
\end{equation}
where  $r_i=\frac{\pi}{6}$, $(x_1,x_2,x_3,x_4,x_5,x_6)=(-\frac{\pi}{4},\frac{\pi}{4},0,\frac{\pi}{2},-\frac{\pi}{2},0)$, $(y_1,y_2,y_3,y_4,y_5,y_6)=(-\frac{\pi}{4},-\frac{\pi}{4},\frac{\pi}{4},\frac{\pi}{4},\frac{\pi}{4},-\frac{3\pi}{4})$ and $z_i=0$ for $i=1,2,\dots ,6$.

In Fig.\,\ref{collision6-3D}, we plot several snapshots of iso-surface  $\{\phi=0\}$ at various time $t=0,0.01,0.02,0.2,0.5,2$ by using the BDF2 scheme of first  approach.  We observe from this figure that the initially separated  spheres connect with each other in  a short time, and eventually merge into a big vesicle. The shape of final steady state  is also observed in \cite{du2006simulating}.

\begin{figure}[htbp]
\centering
\includegraphics[width=0.32\textwidth,clip==]{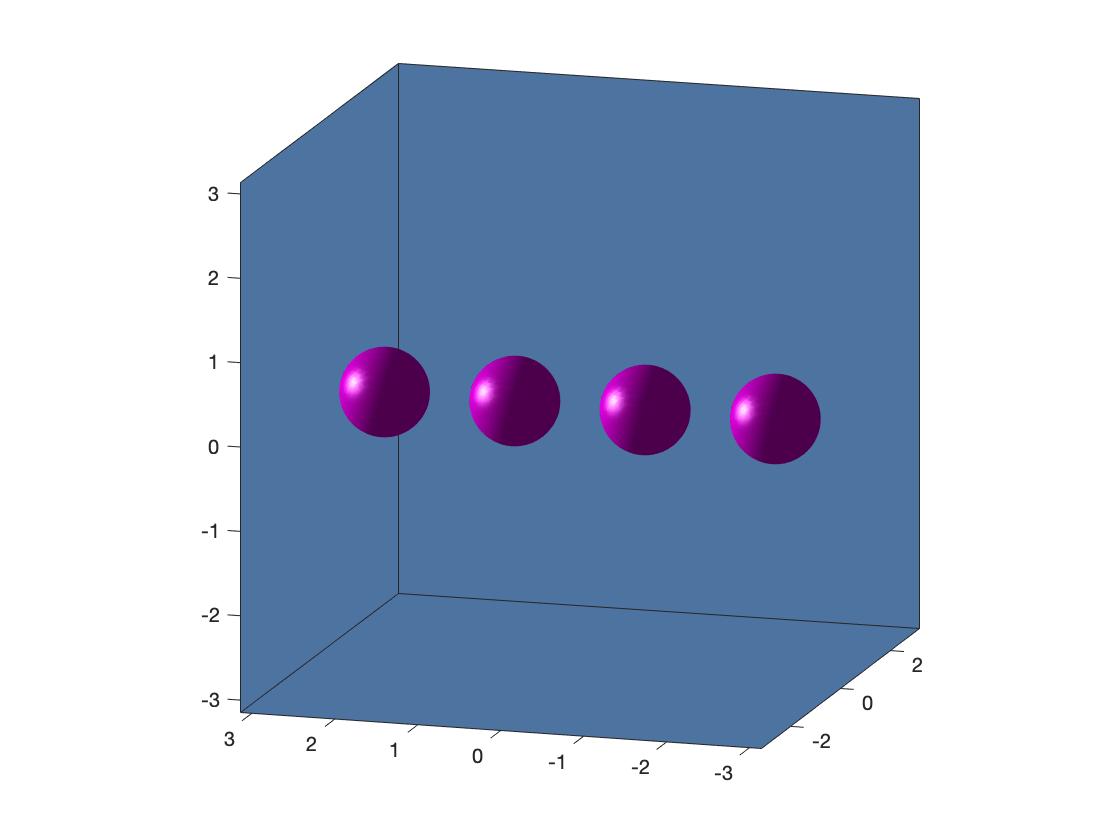}
\includegraphics[width=0.32\textwidth,clip==]{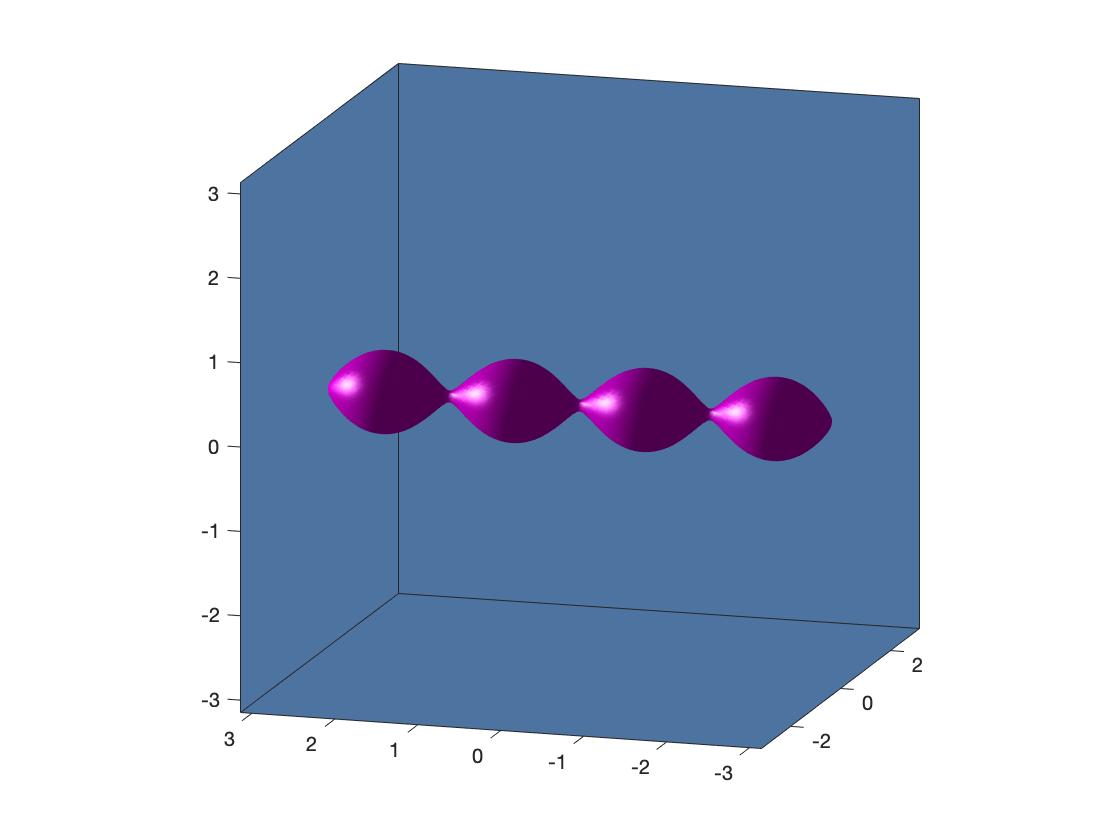}
\includegraphics[width=0.32\textwidth,clip==]{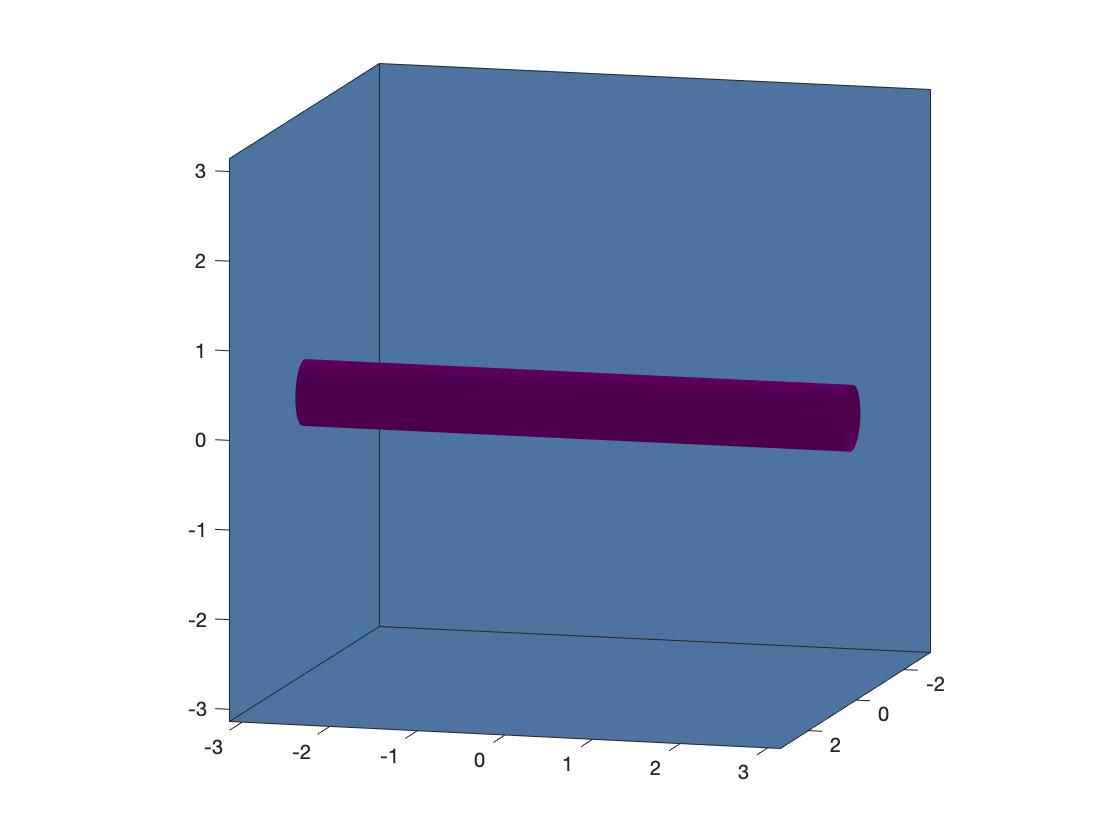}
\caption{Collision of four close-by spherical vesicles by  using the third approach  with  time step size $\delta t= 1\times 10^{-4}$. Snapshots of  isosurface
of $\{\phi=0\}$ at t = 0, 0.02, 1.}\label{cylinder4-3D}
\end{figure}

\begin{figure}[htbp]
\centering
\includegraphics[width=0.32\textwidth,clip==]{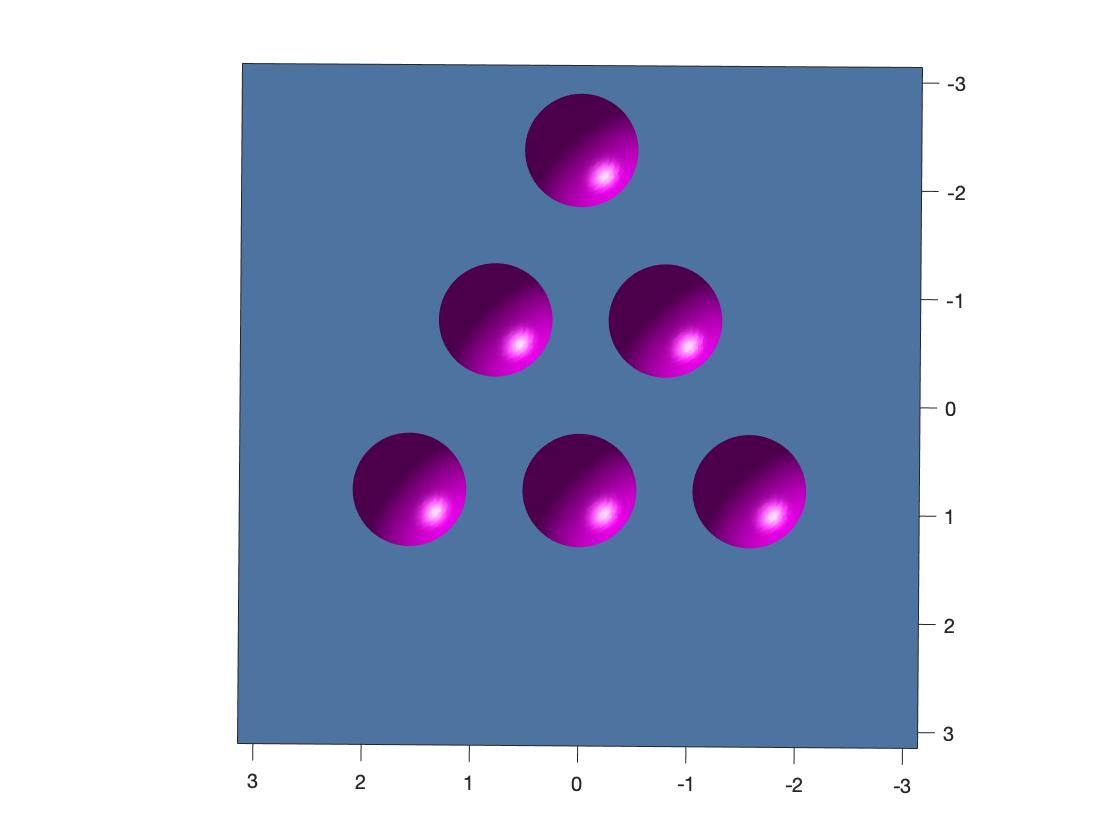}
\includegraphics[width=0.32\textwidth,clip==]{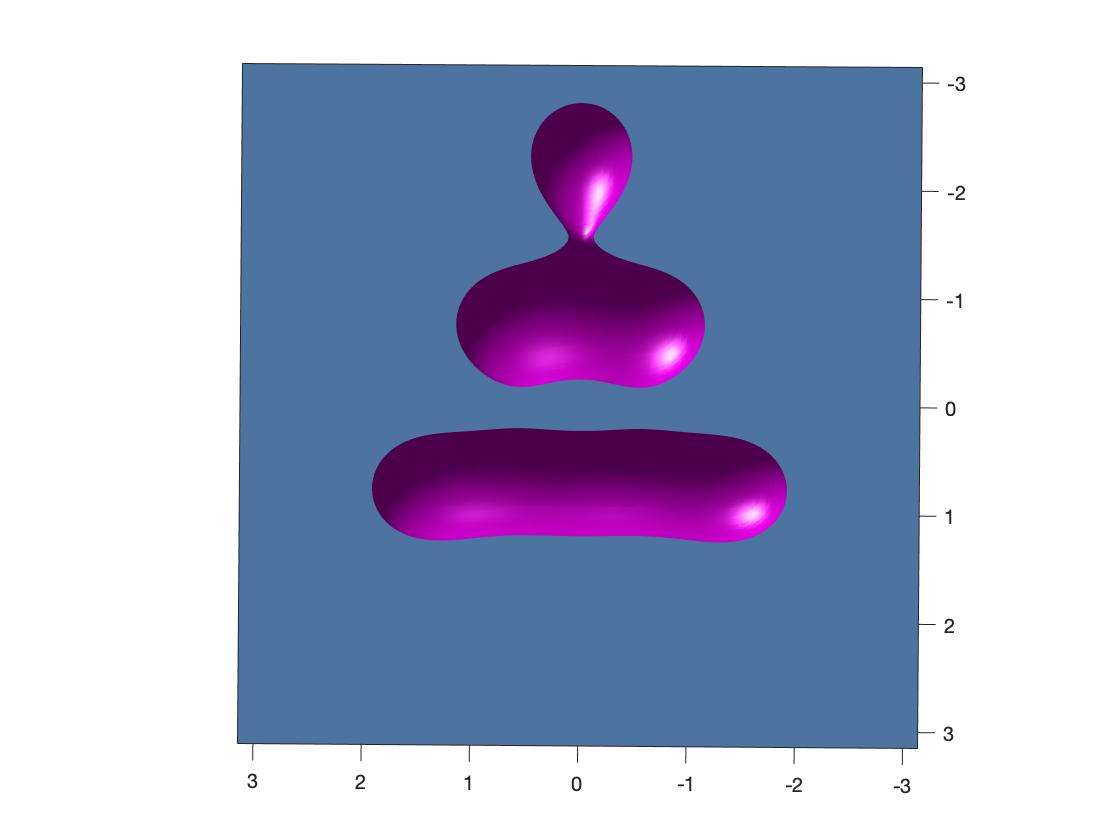}
\includegraphics[width=0.32\textwidth,clip==]{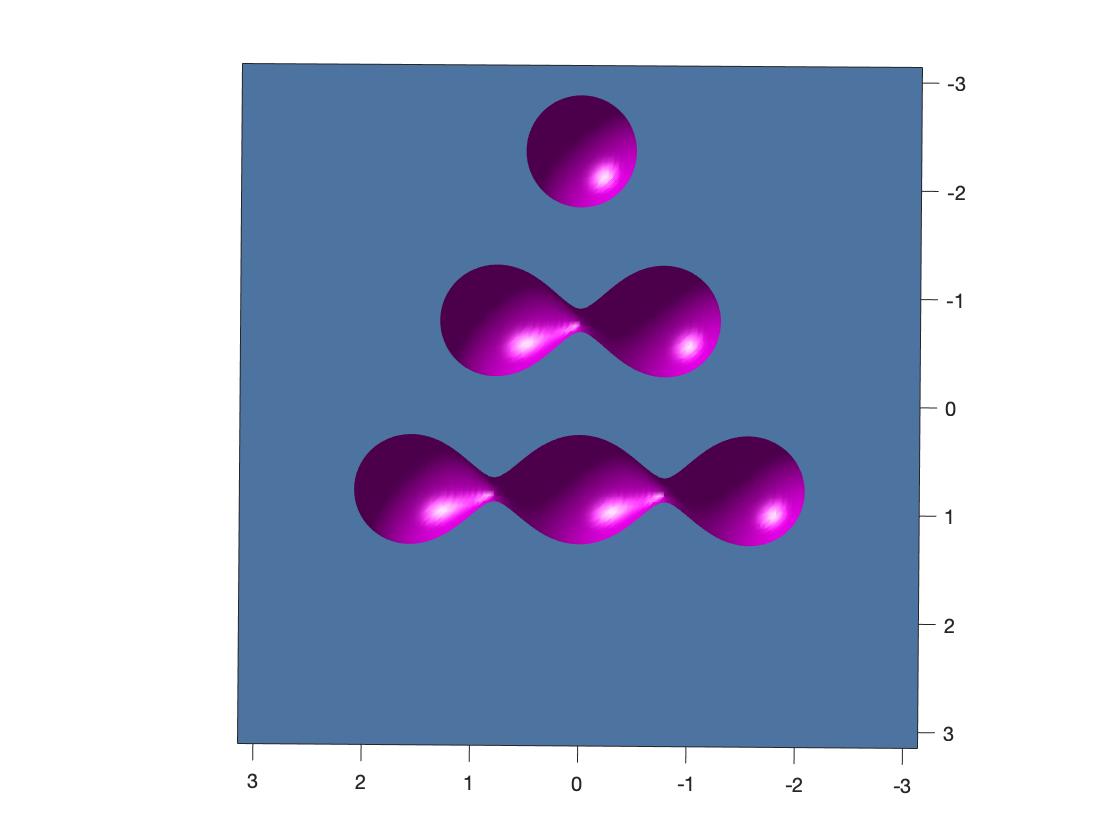}
\includegraphics[width=0.32\textwidth,clip==]{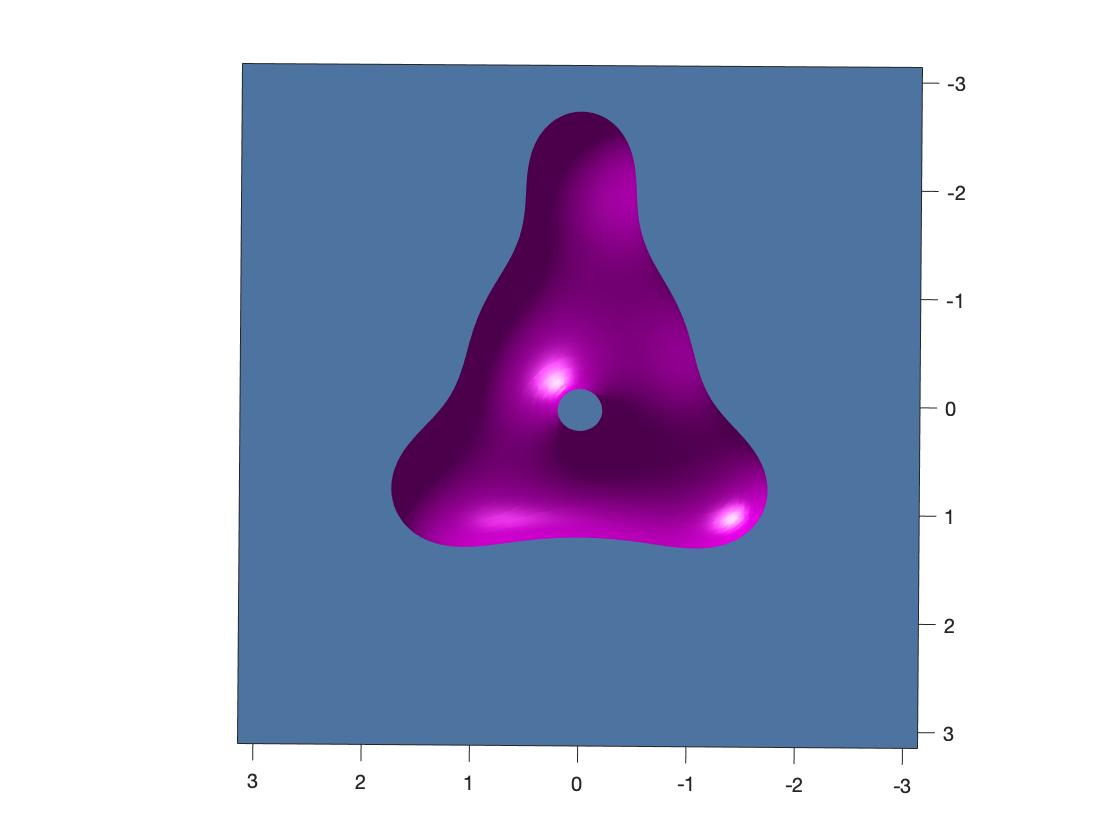}
\includegraphics[width=0.32\textwidth,clip==]{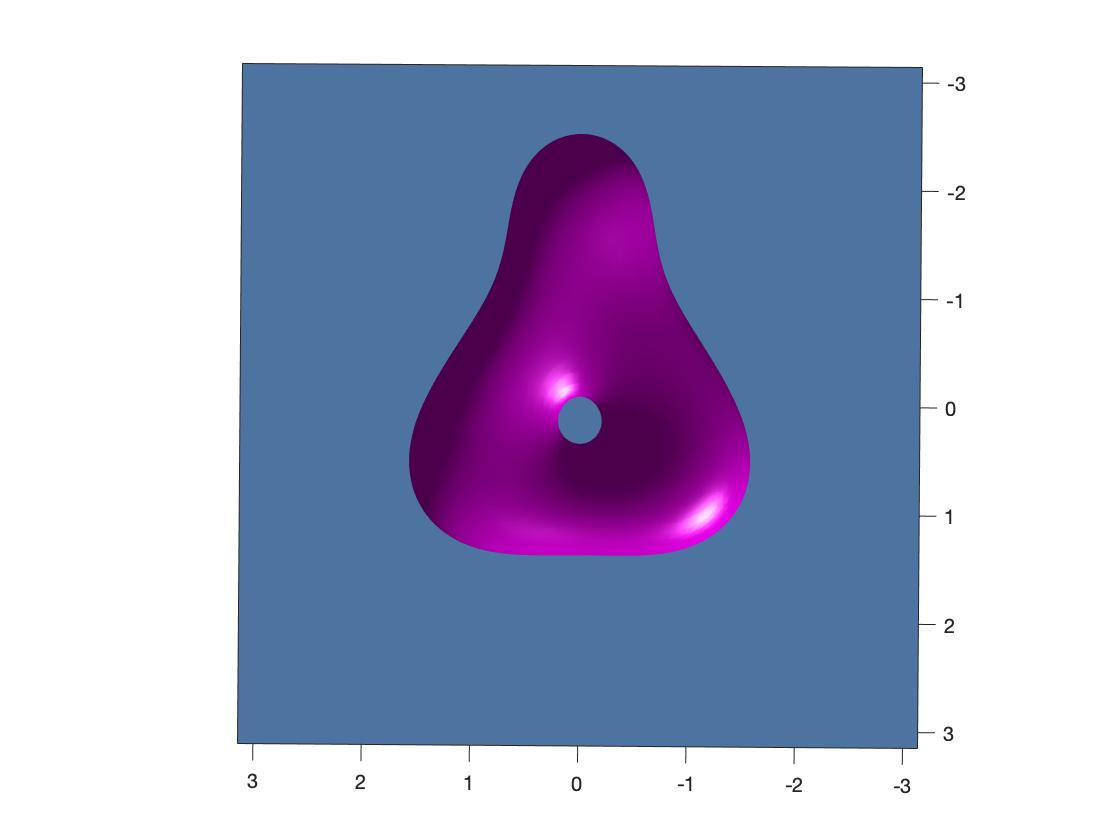}
\includegraphics[width=0.32\textwidth,clip==]{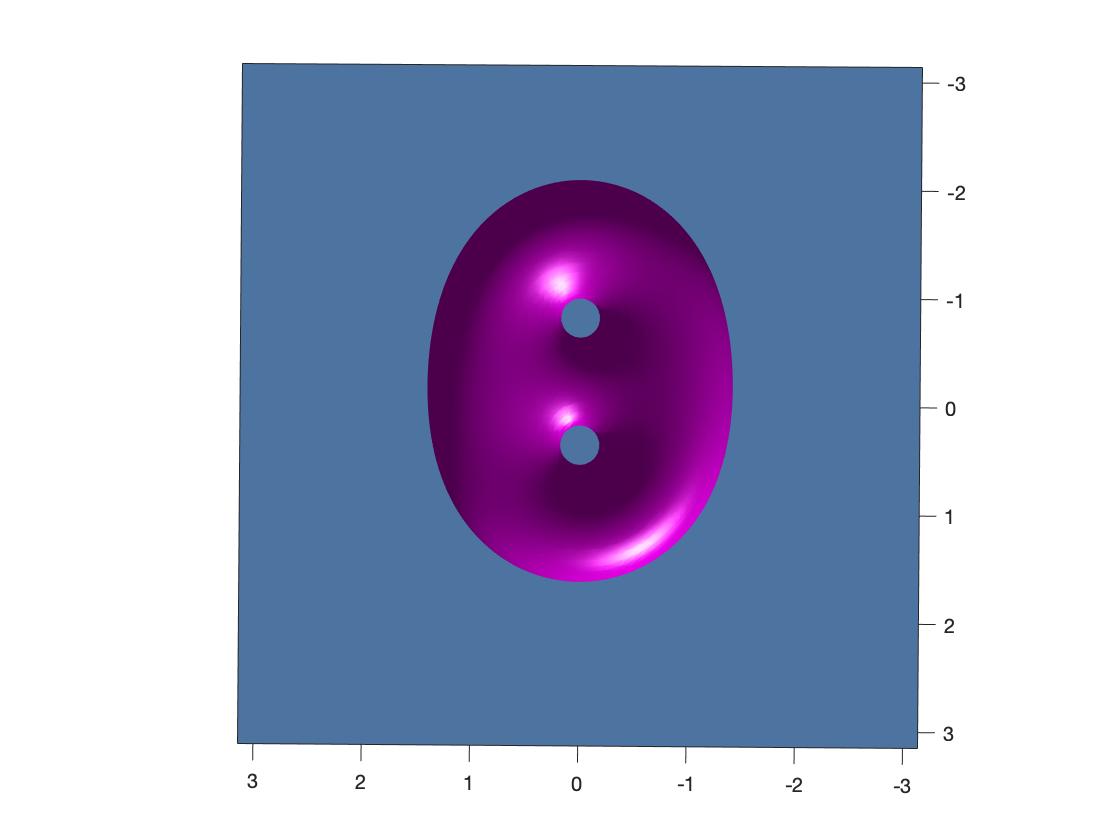}
\caption{Collision of six close-by spherical vesicles  by using the first approach with  time step size $\delta t= 1\times 10^{-4}$. Snapshots of  isosurface
of $\{\phi=0\}$  at t = 0, 0.01,0.02, 0.2,0.5,2.}\label{collision6-3D}
\end{figure}

 We also plot, in Fig.\,\ref{vesicle_lagrange},  the evolution of Lagrange multiplier $\lambda $  for these two examples.  We observe that the Lagrange multiplier $\lambda$ will change rapidly at the begining and gradually settle down to a steady state value. We also observe that  $\lambda$ can become  negative.

\begin{figure}[htbp]
\centering
\includegraphics[width=0.45\textwidth,clip==]{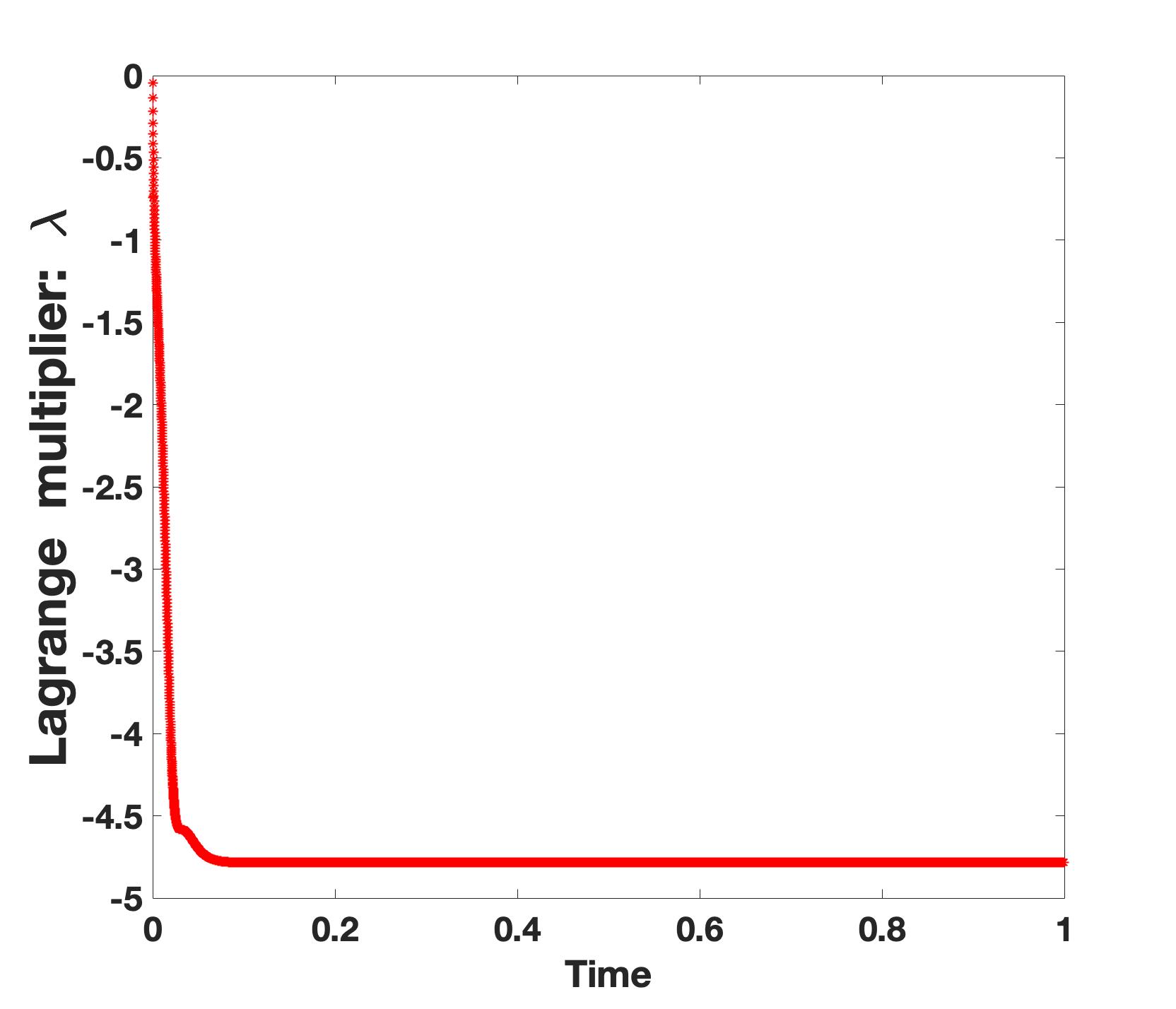}
\includegraphics[width=0.45\textwidth,clip==]{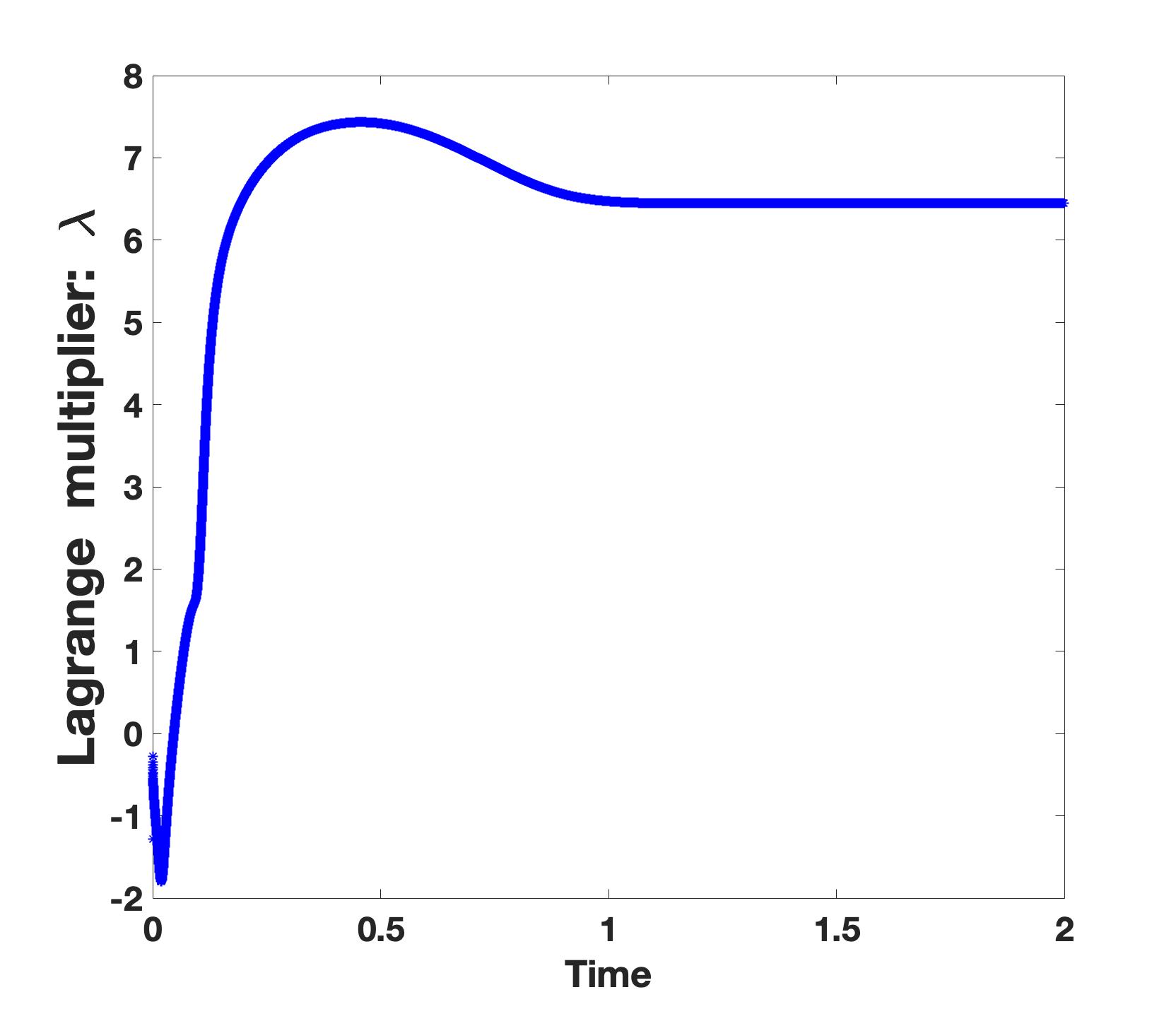}
\caption{Evolution of Lagrange multiplier  $\lambda$ for the examples in  Fig.\,\ref{cylinder4-3D} and Fig.\,\ref{collision6-3D}. }\label{vesicle_lagrange}
\end{figure}


\subsection{Optimal partition model}
We  present below numerical experiments  for the optimal  partition problem \eqref{part:1}-\eqref{part:3}.  The computational domain is set to  $\Omega = [-\pi,\pi)^2$.  The boundary condition is periodic and  the Fourier-spectral method is adopted to discretize the space variables.  In all computations,  we use  $128^2$ Fourier modes with interfacial width parameter $\epsilon=0.01$.  To better visualize the subdomain evolution, we assign an integer valued marker function $\chi_i$ which equals to $i$ in the region $i$, and $\chi_i=0$ in other regions.  The initial condition for $\phi_i$ is  set to be the marker function $\chi_i$.  The BDF2 scheme of first approach with time step $\delta t=10^{-5}$ is used for all examples below.

For the  first example,  we take  $m=4$ with four connected quadrilaterals  as the  initial configuration.  In  Fig.\,\ref{LGM-partition4-2D}, the  evolutions  of the phase configuration at various times  are depicted. We observe  that patterns in the partition  gradually evolve into  hexagonal patterns as the final steady state.

For the optimal partition  problem,  it is shown in \cite{CL1}  that all Lagrange multipliers   are positive and will decay with time.  In  Fig.\,\ref{part4_lagrange}, we plot   evolutions of the  four Lagrange multipliers  and observe that they are indeed positive and decay with time. 

Next,   we increase the numbers of partitions to $m=8$, and plot in  Fig.\,\ref{LGM-partition8-2D} the  evolutions  of the phase configuration at various times. We observe that 
 the  partition eventually evolves into a honeycomb  shape with mostly hexagonal patterns. Similar behaviors are observed with $m=10$ as shown in Fig.\,\ref{LGM-partition10-2D}.
 
 These numerical results   are consistent with the numerical  simulations presented in \cite{du2008numerical}.

\begin{figure}[htbp]
\centering
\subfigure[t=0]{
\includegraphics[width=0.30\textwidth,clip==]{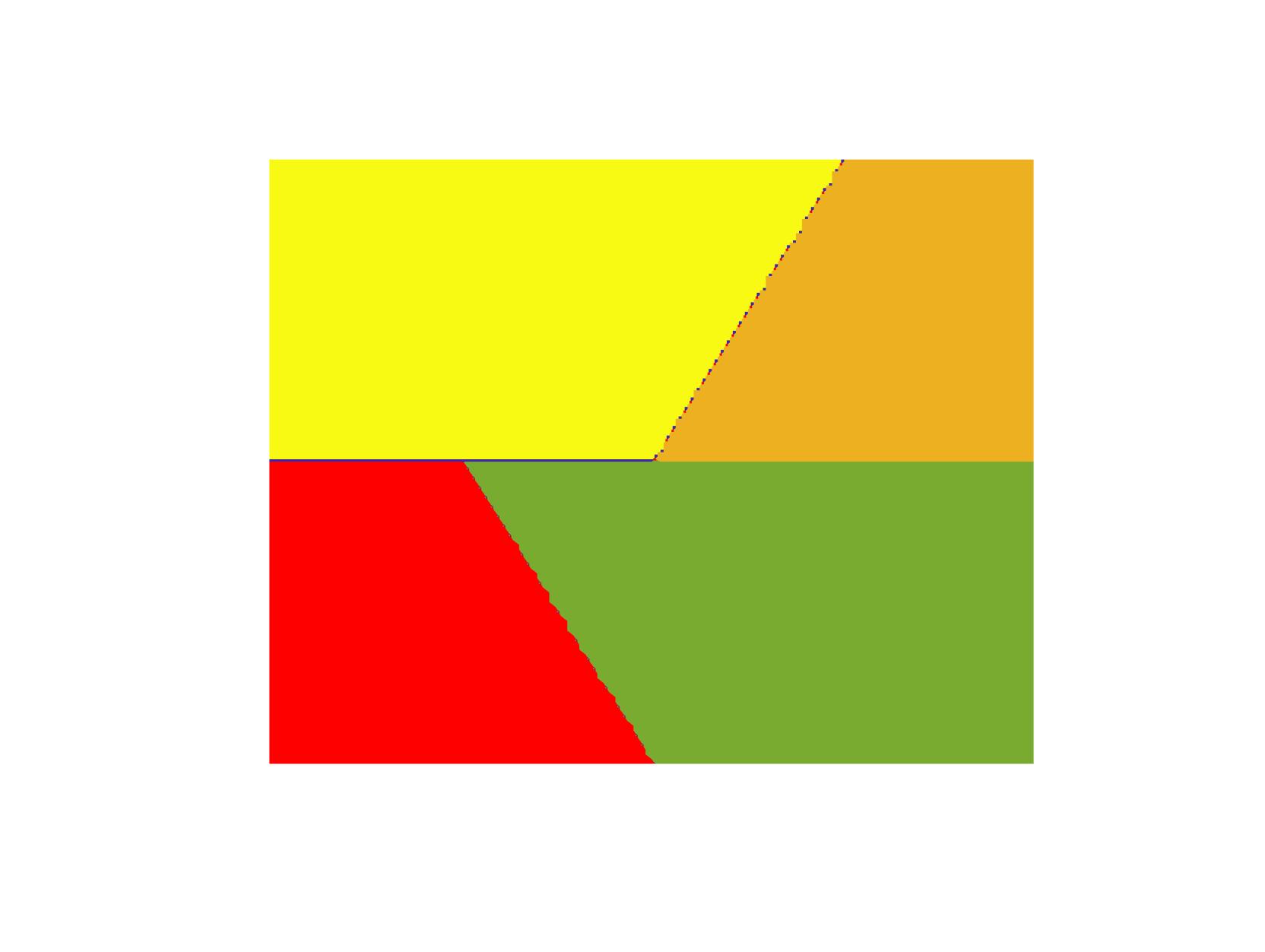}}
\subfigure[t=0.05]{
\includegraphics[width=0.30\textwidth,clip==]{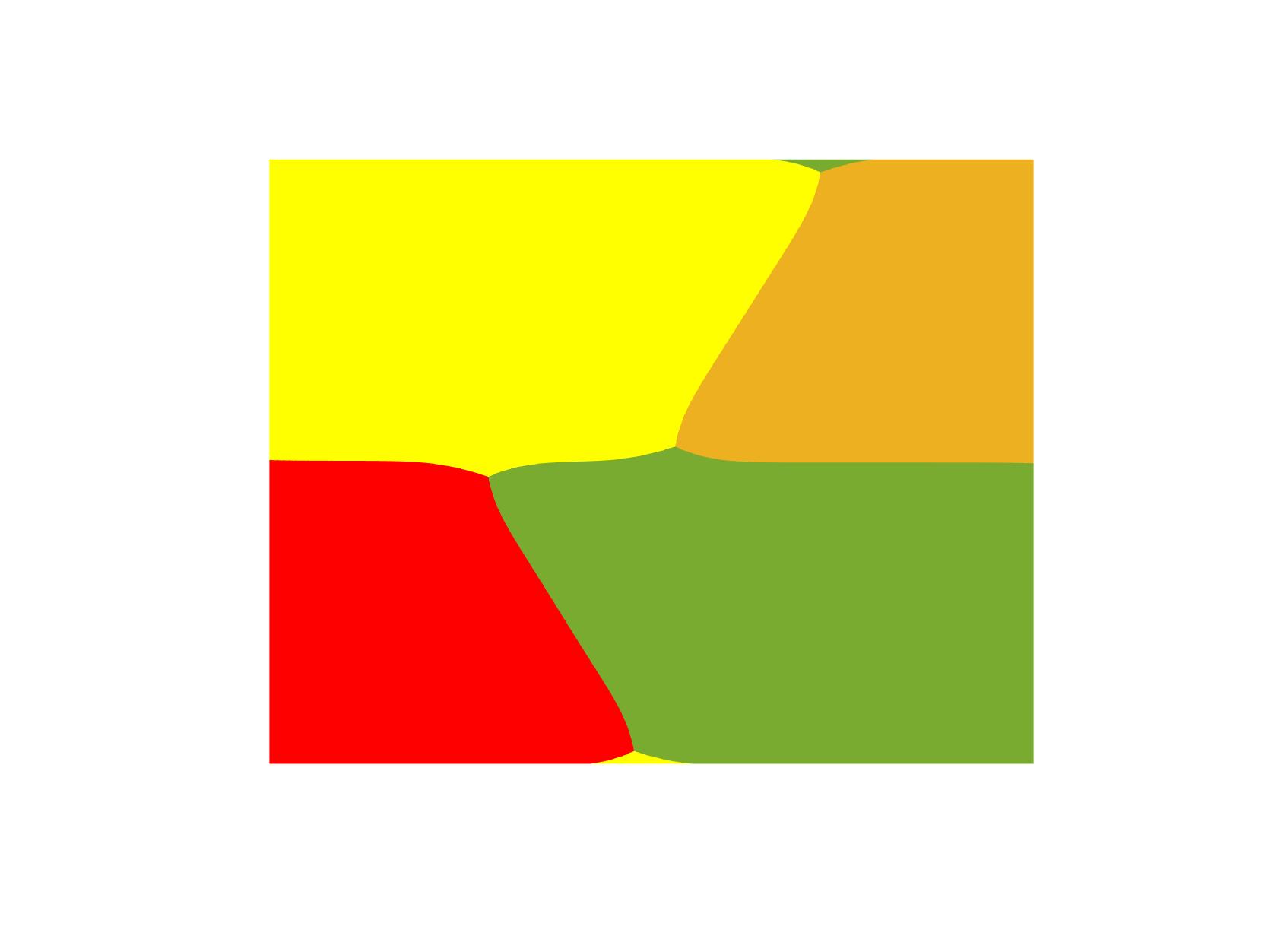}}
\subfigure[t=0.5]{
\includegraphics[width=0.30\textwidth,clip==]{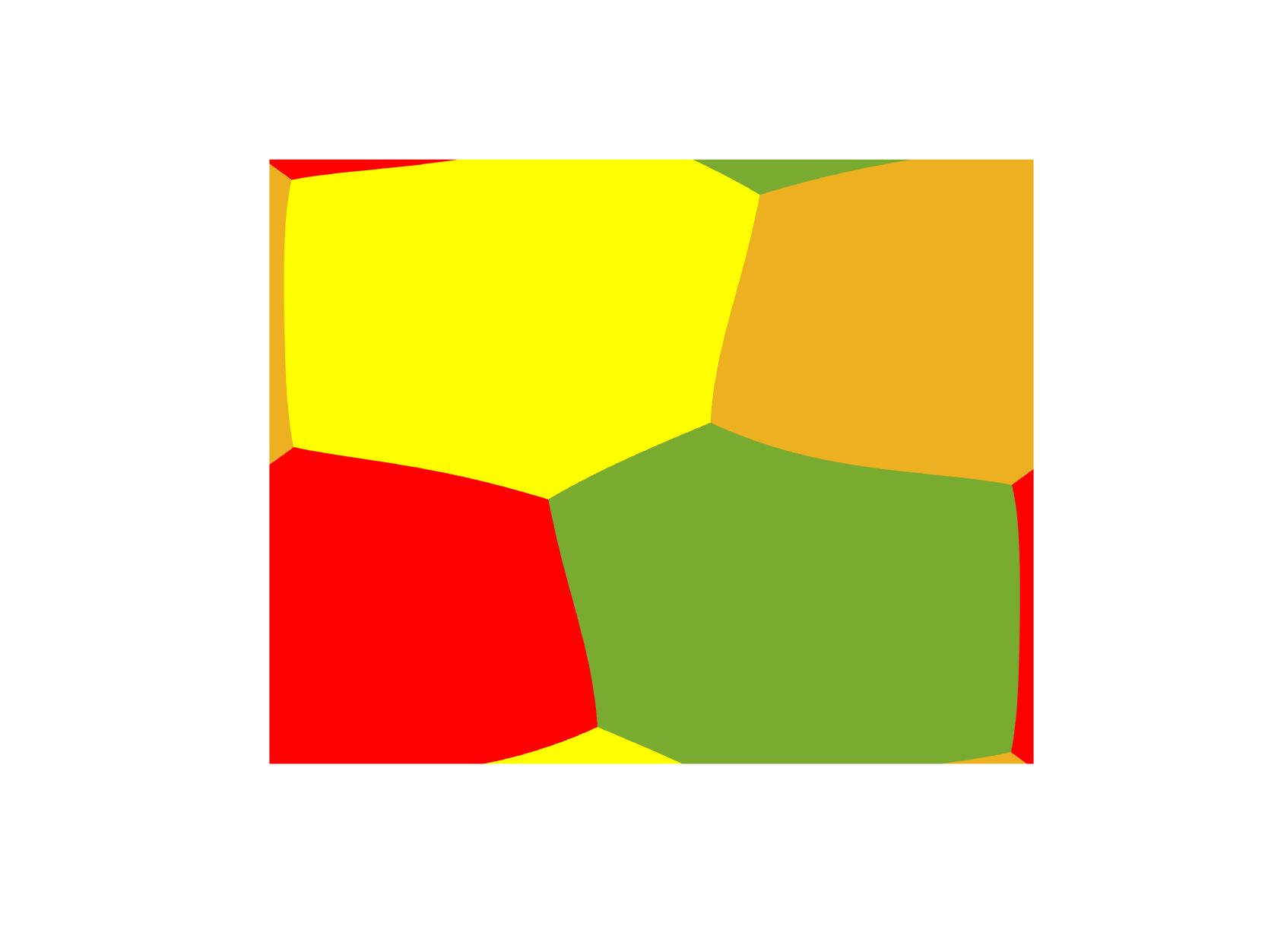}}
\subfigure[t=1]{
\includegraphics[width=0.30\textwidth,clip==]{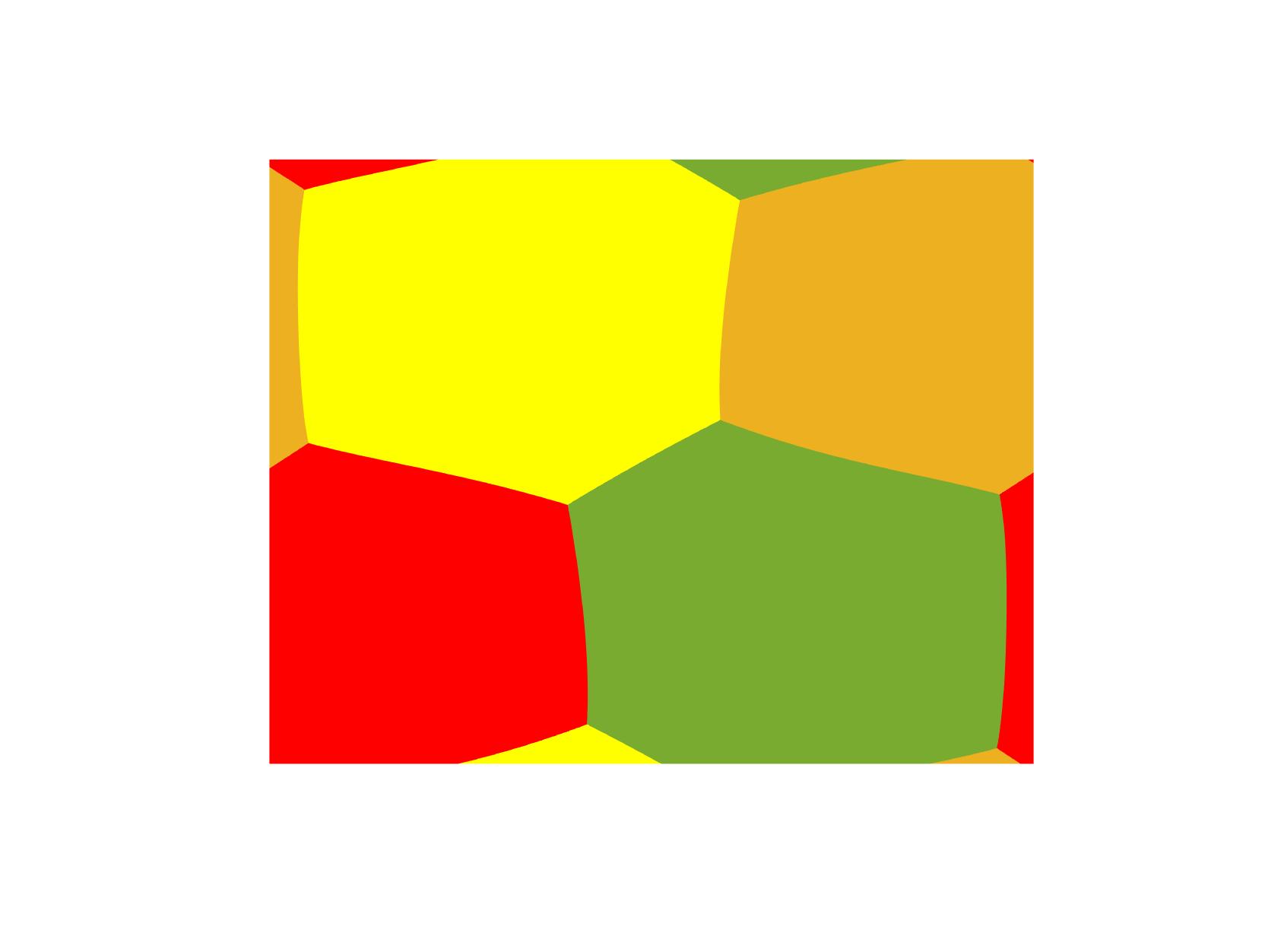}}
\subfigure[t=5]{
\includegraphics[width=0.30\textwidth,clip==]{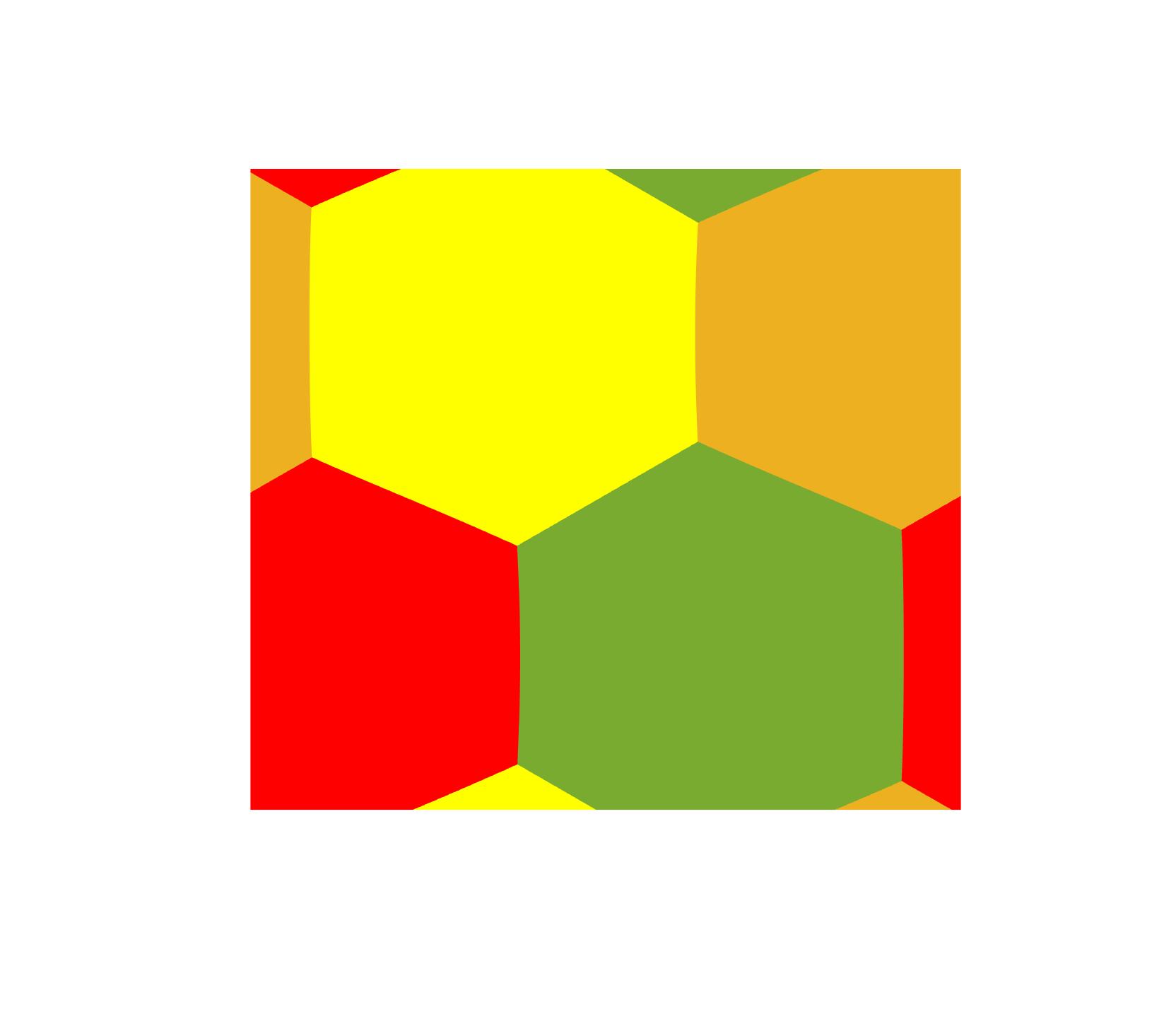}}
\subfigure[t=10]{
\includegraphics[width=0.30\textwidth,clip==]{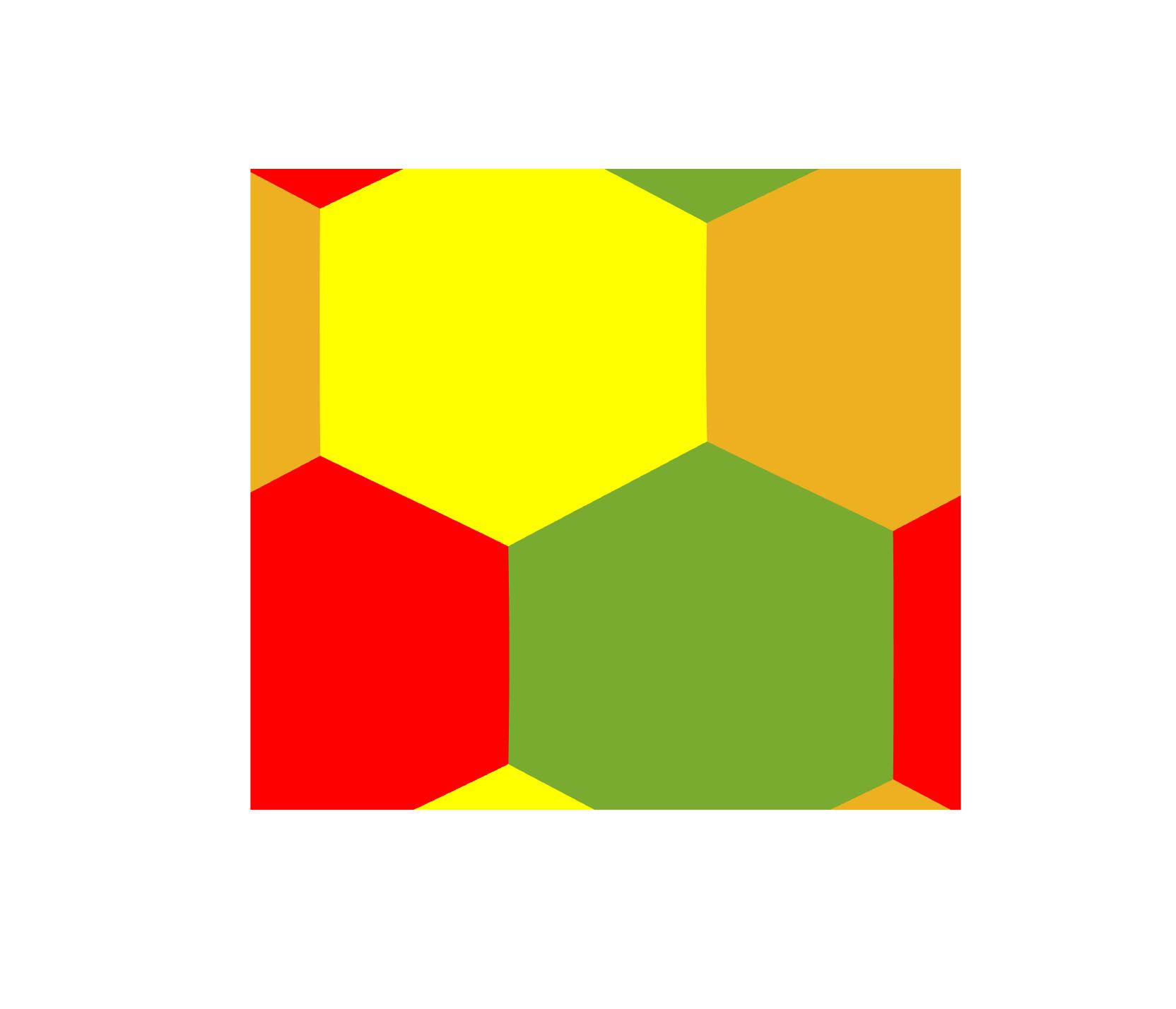}}
\caption{A 4-subdomain partition: initial partition and subdomains at times $t=0,0.05,0.5,1,5,10$ computed by the BDF2 scheme of first approach  with $\delta t=1\times 10^{-5}$.}\label{LGM-partition4-2D}
\end{figure}

\begin{figure}[htbp]
\centering
\includegraphics[width=0.45\textwidth,clip==]{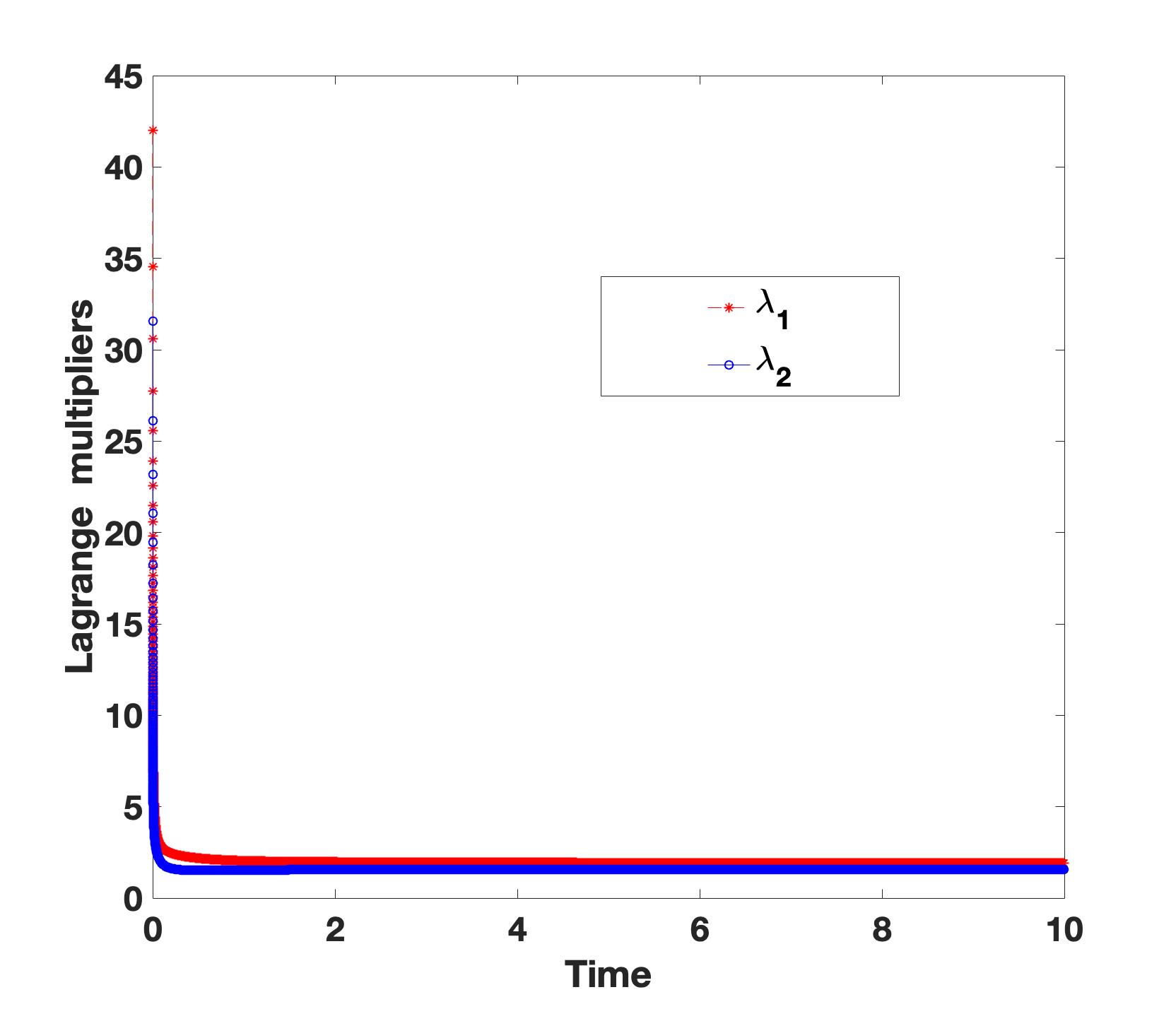}
\includegraphics[width=0.45\textwidth,clip==]{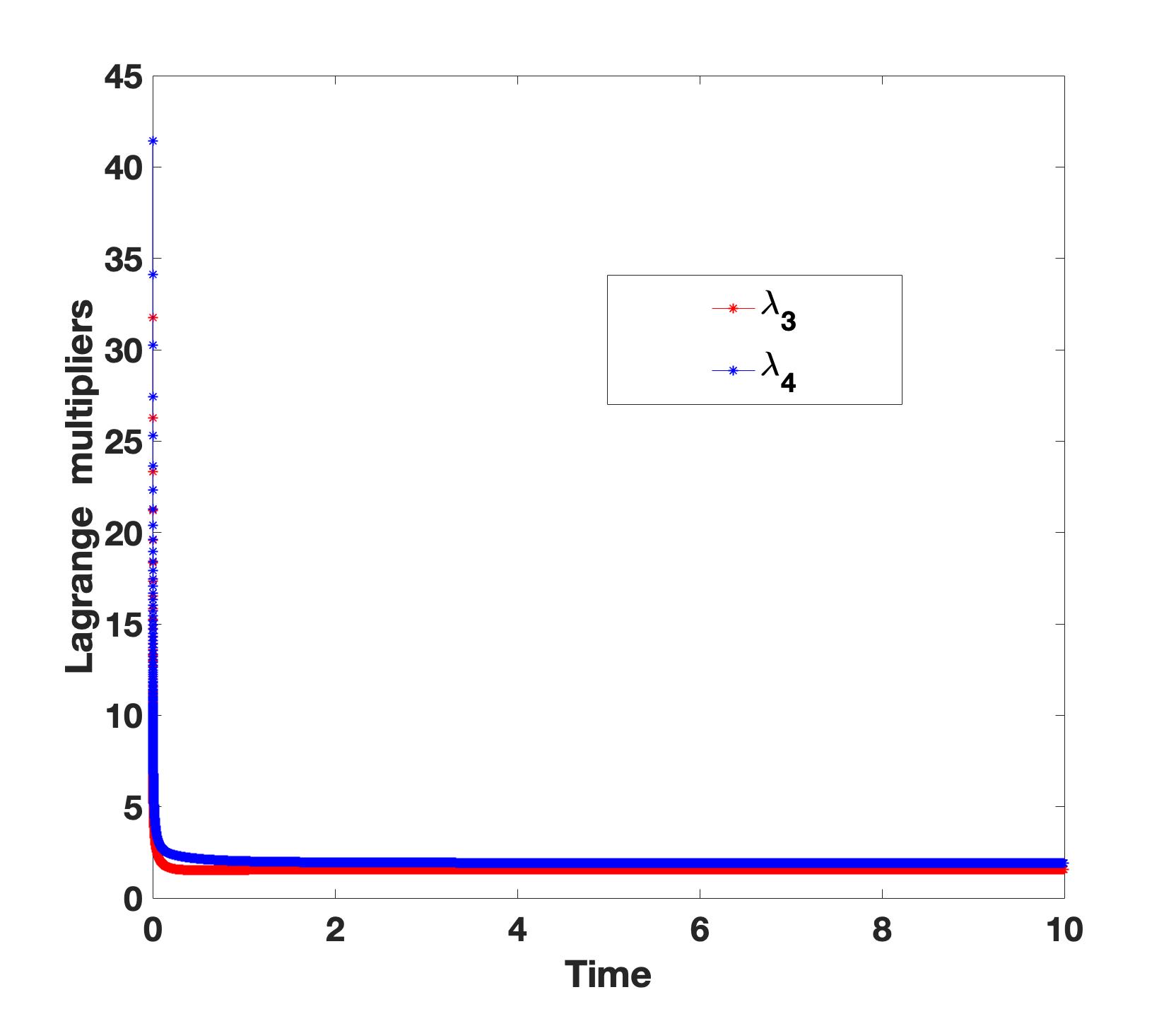}
\caption{Evolution of Lagrange Multipliers $\lambda_1,\lambda_2$ and $\lambda_3,\lambda_4$ with respect to time for 4-subdomain partition in Fig,\,\ref{LGM-partition4-2D}.}\label{part4_lagrange}
\end{figure}

\begin{figure}[htbp]
\centering
\subfigure[t=0]{
\includegraphics[width=0.30\textwidth,clip==]{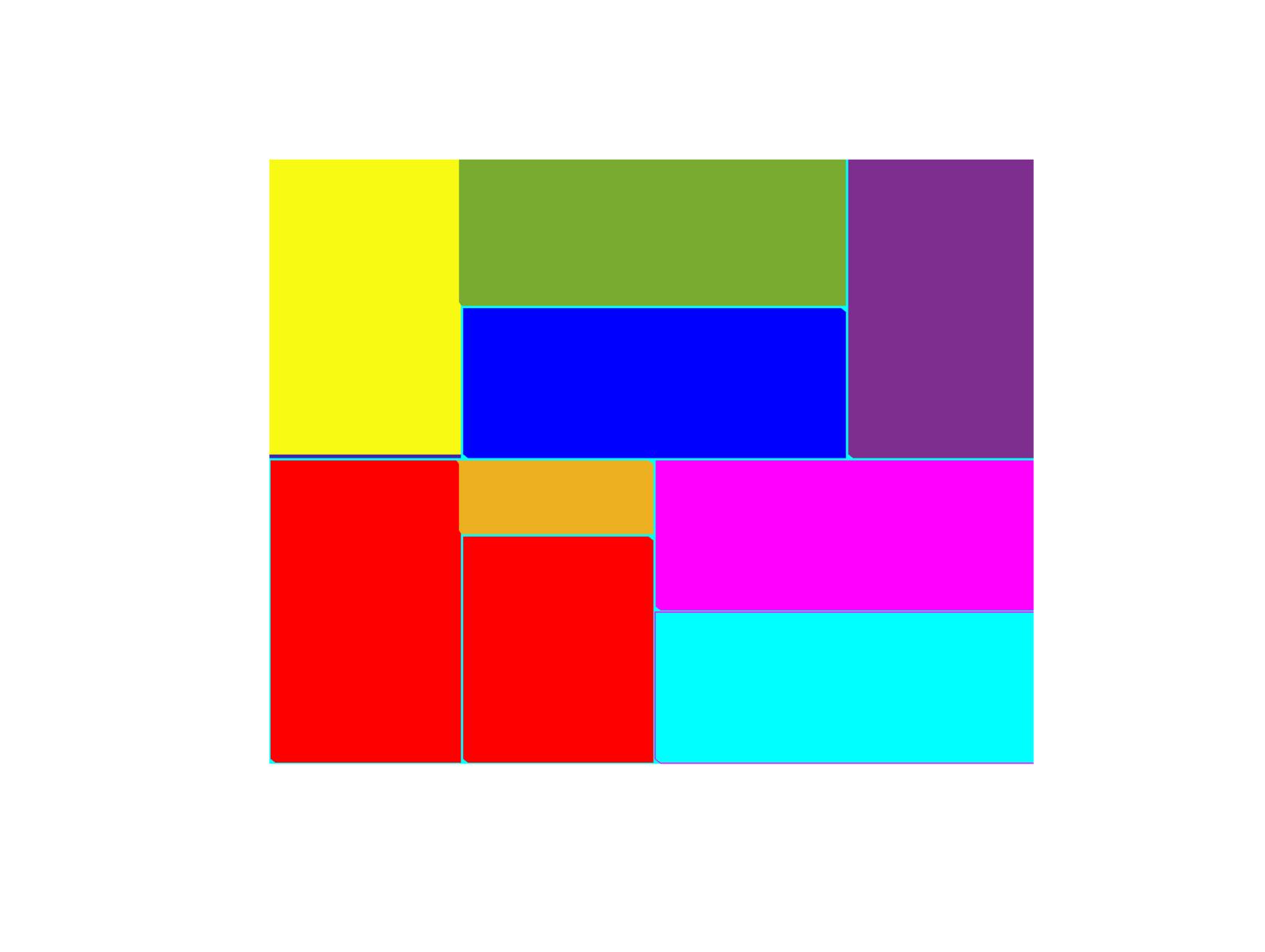}}
\subfigure[t=0.05]{
\includegraphics[width=0.30\textwidth,clip==]{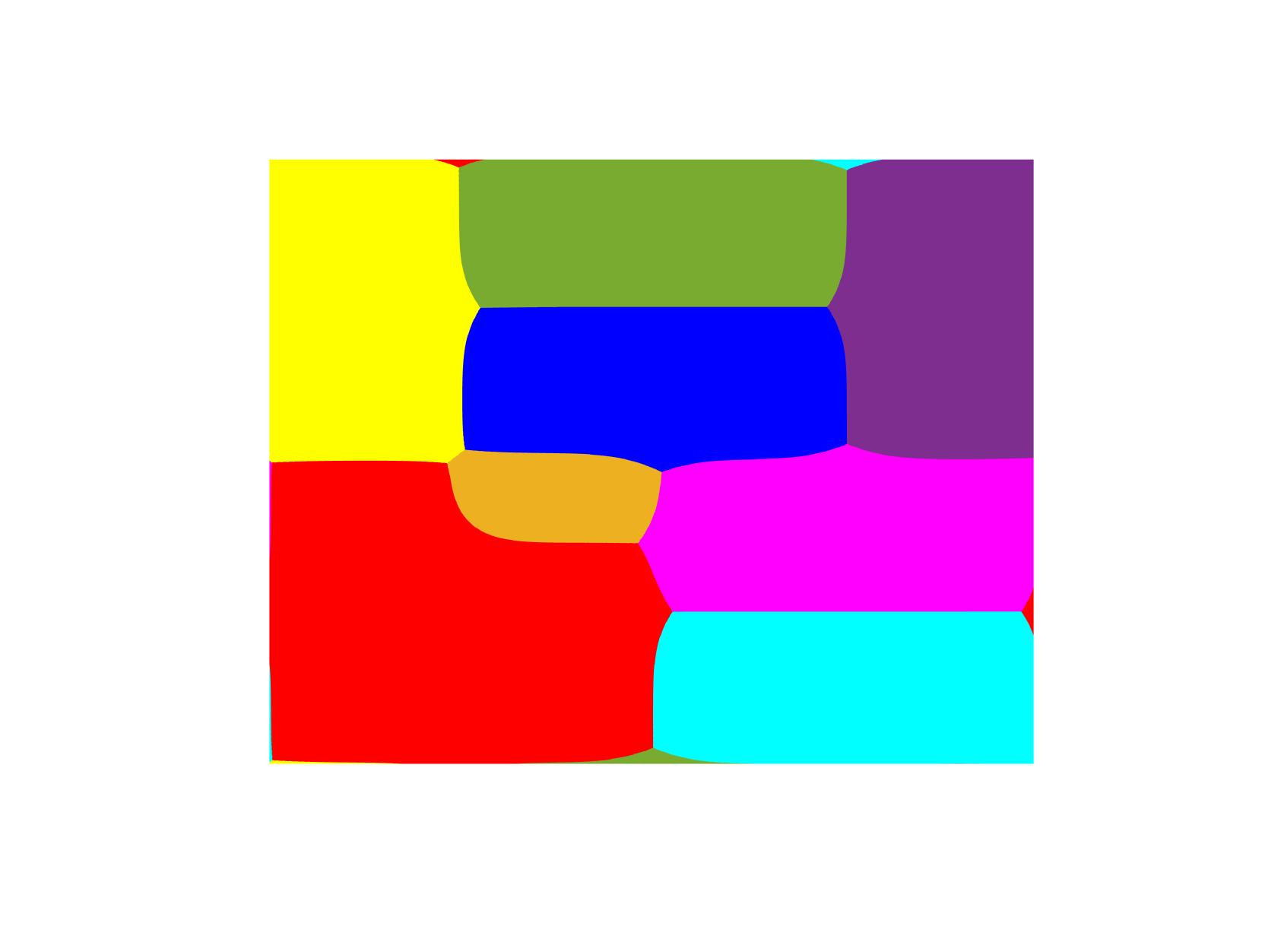}}
\subfigure[t=0.5]{
\includegraphics[width=0.30\textwidth,clip==]{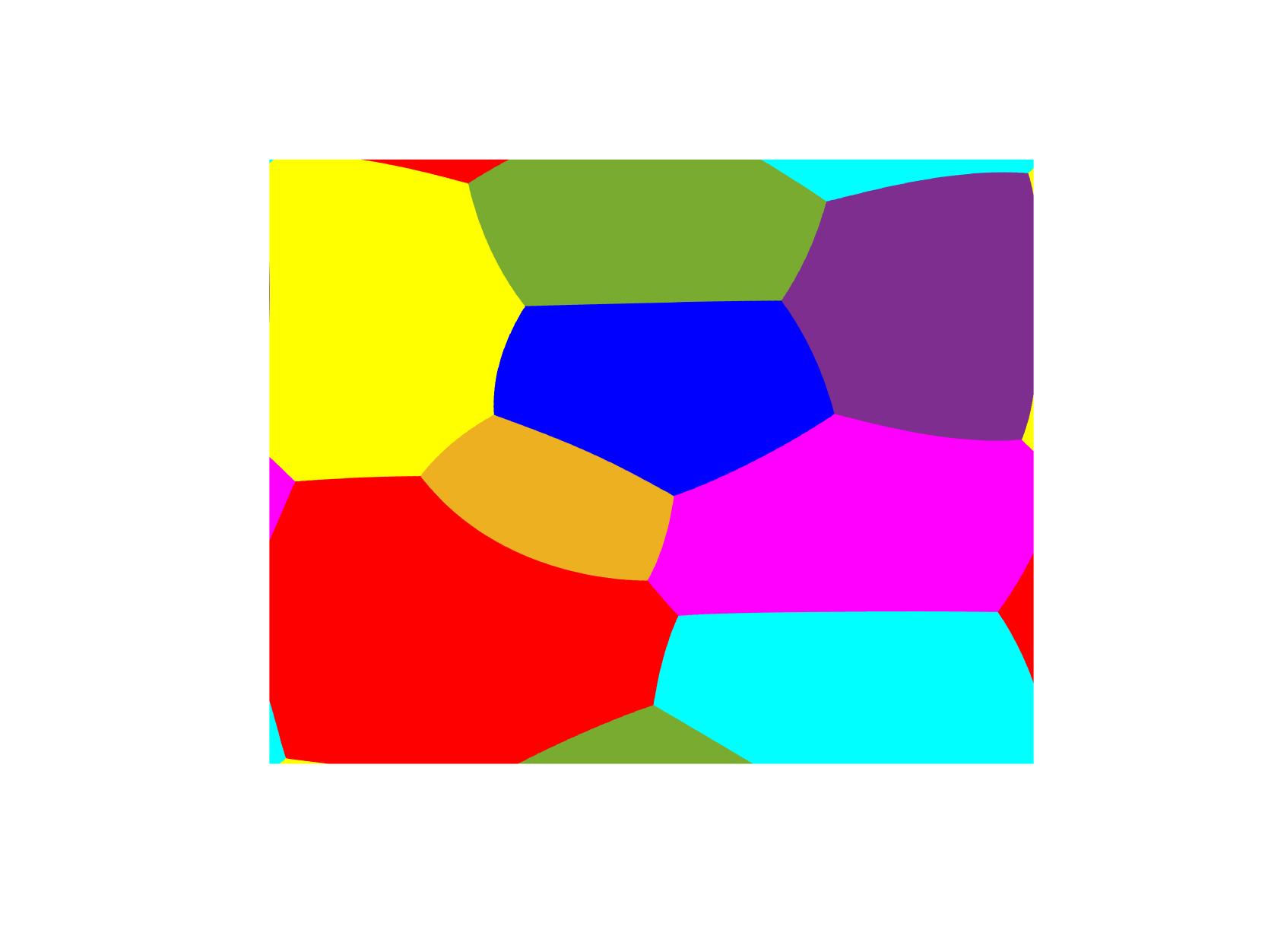}}
\subfigure[t=2]{
\includegraphics[width=0.30\textwidth,clip==]{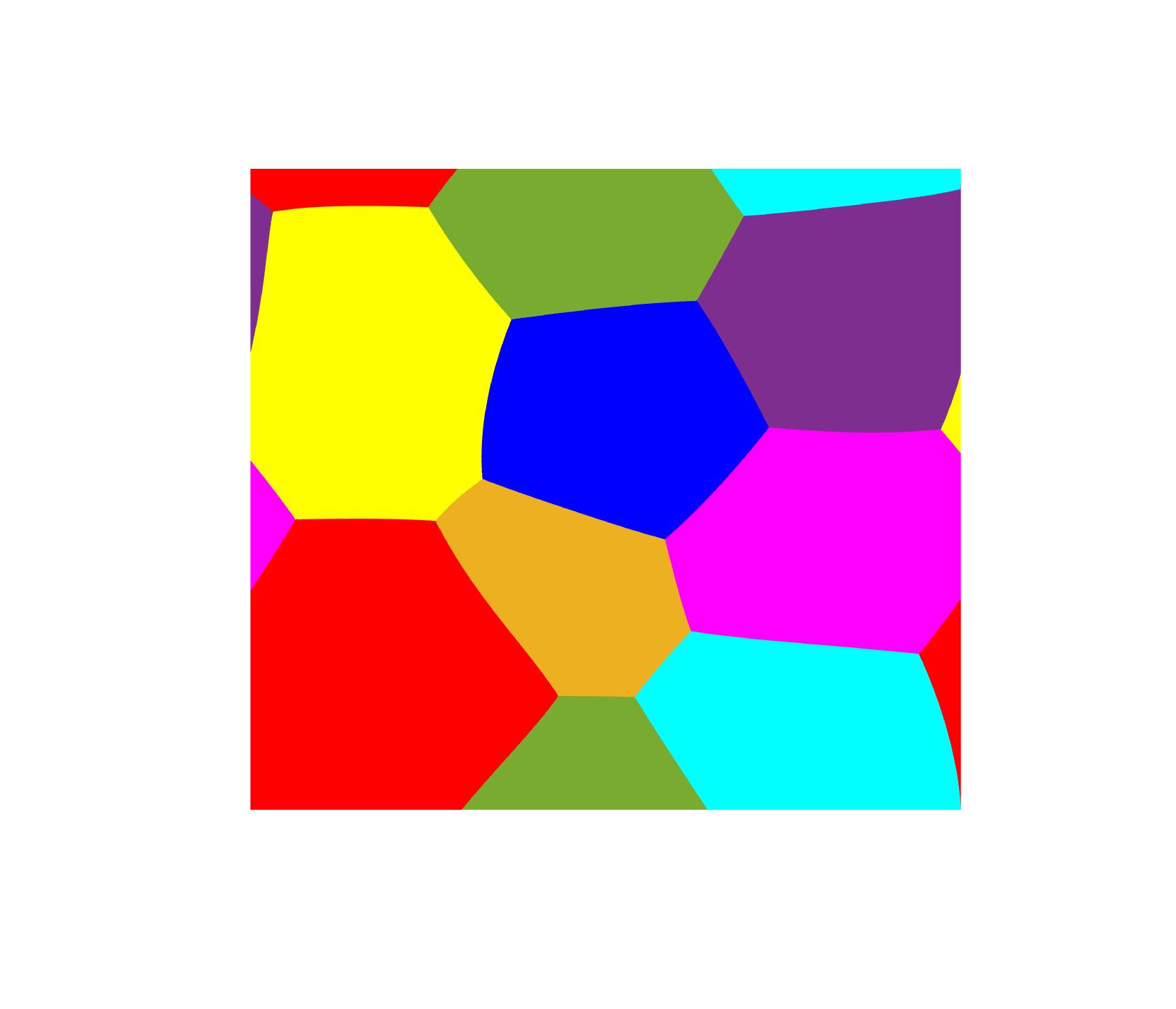}}
\subfigure[t=5]{
\includegraphics[width=0.30\textwidth,clip==]{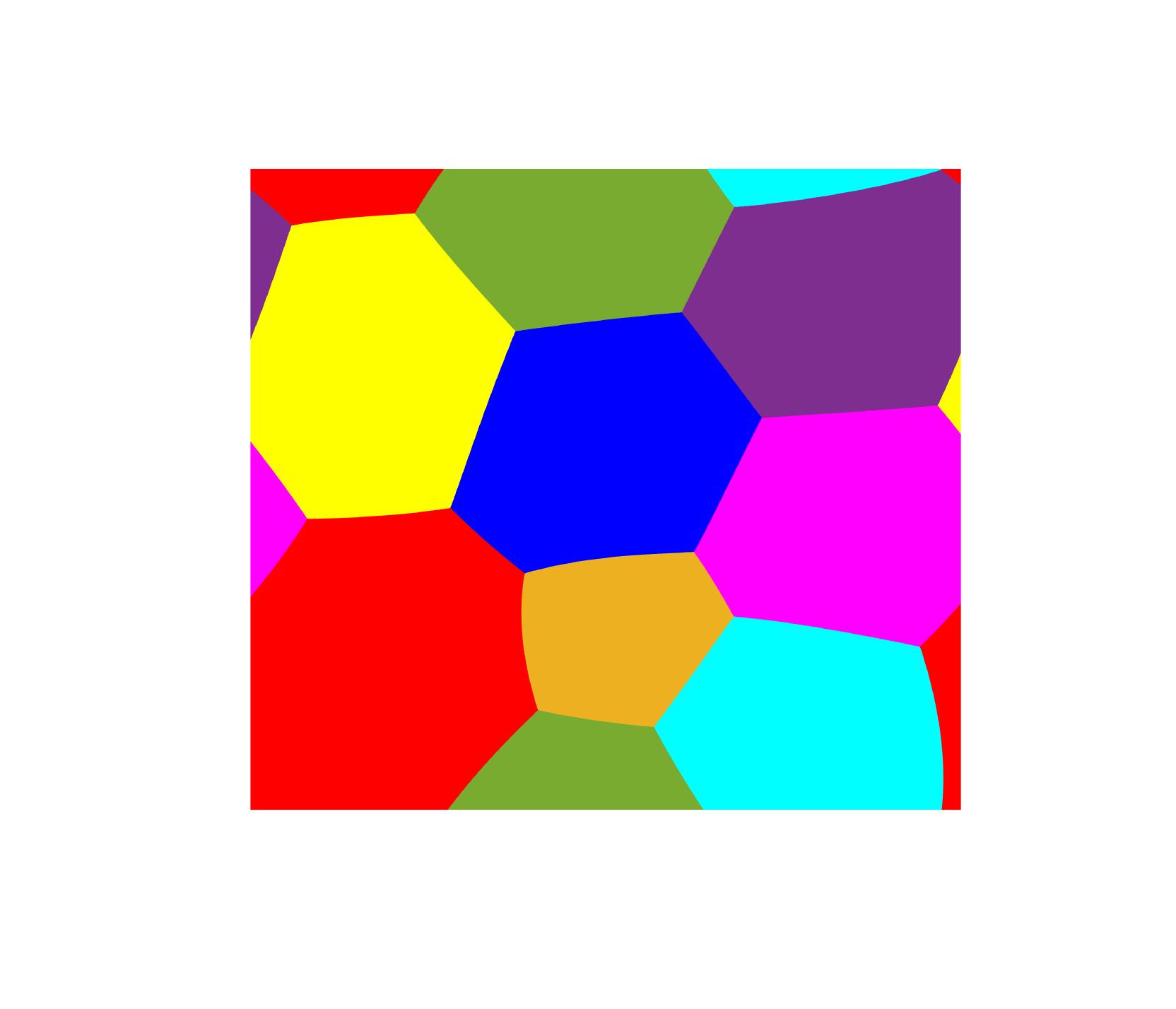}}
\subfigure[t=10]{
\includegraphics[width=0.30\textwidth,clip==]{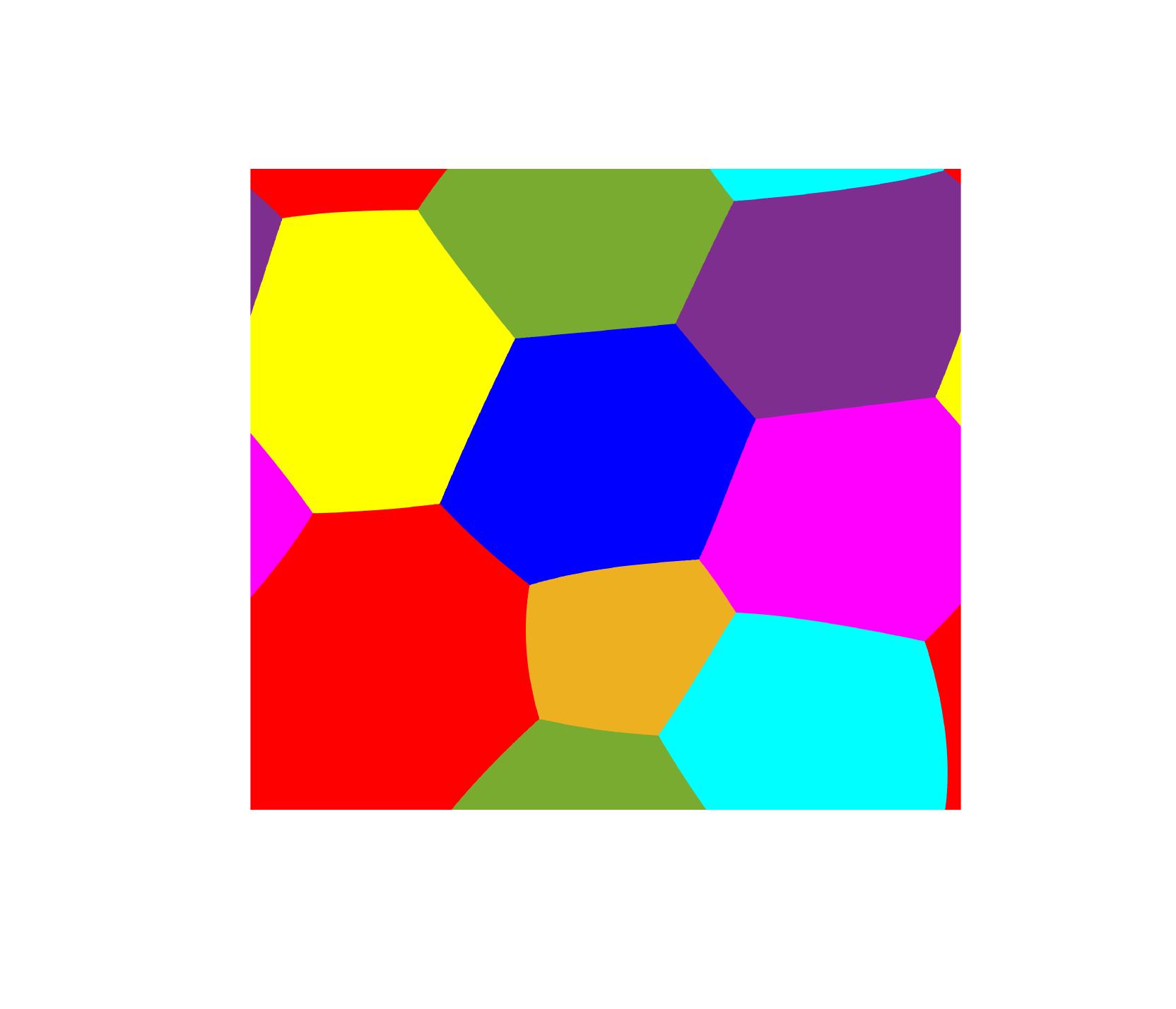}}
\caption{A 8-subdomain partition: initial partition and subdomains at times $t=0,0.05,0.5,2,5,10$ computed by the BDF2 scheme of first approach  with $\delta t=1\times 10^{-5}$.}\label{LGM-partition8-2D}
\end{figure}

\begin{figure}[htbp]
\centering
\subfigure[t=0]{
\includegraphics[width=0.30\textwidth,clip==]{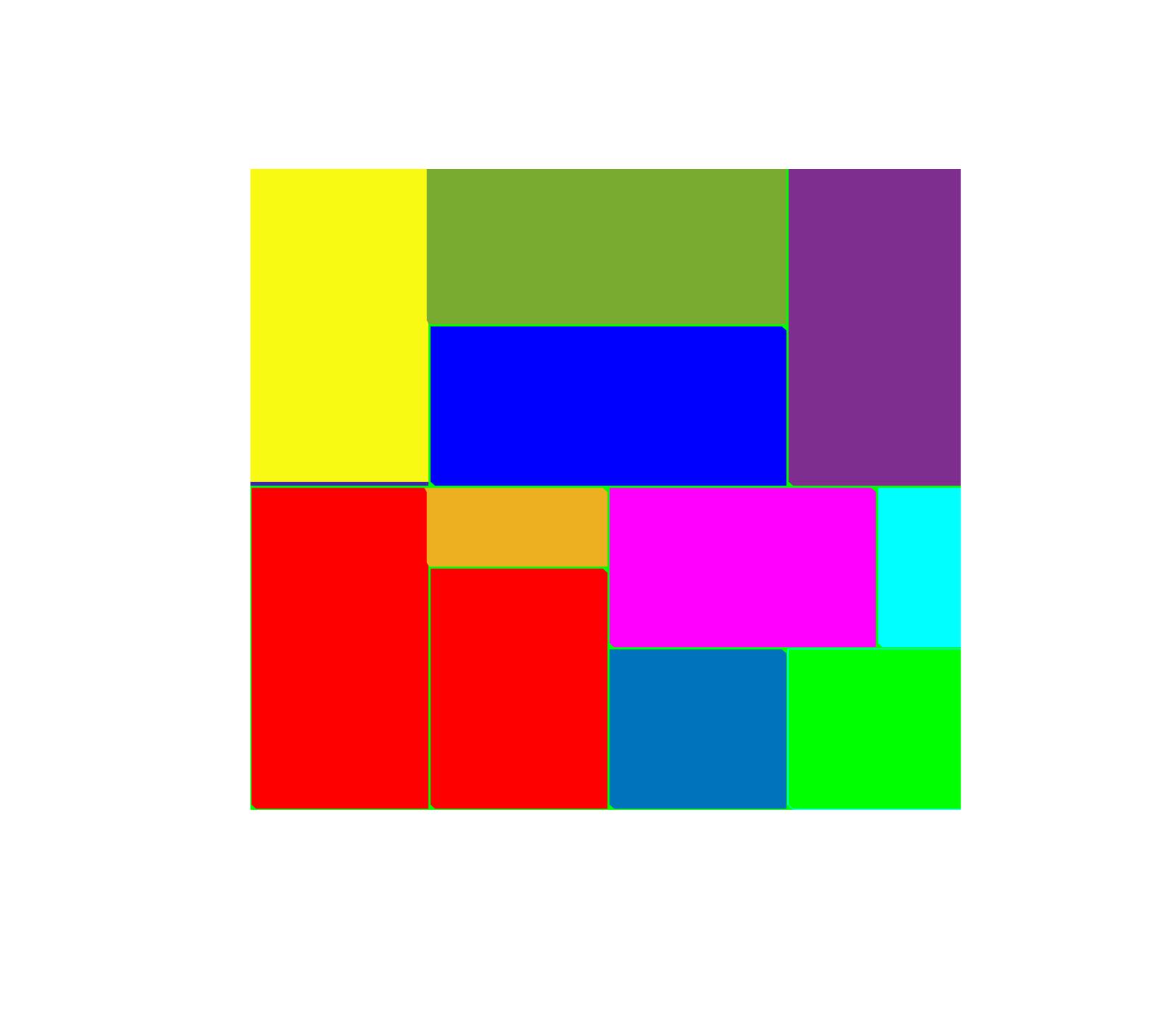}}
\subfigure[t=0.05]{
\includegraphics[width=0.30\textwidth,clip==]{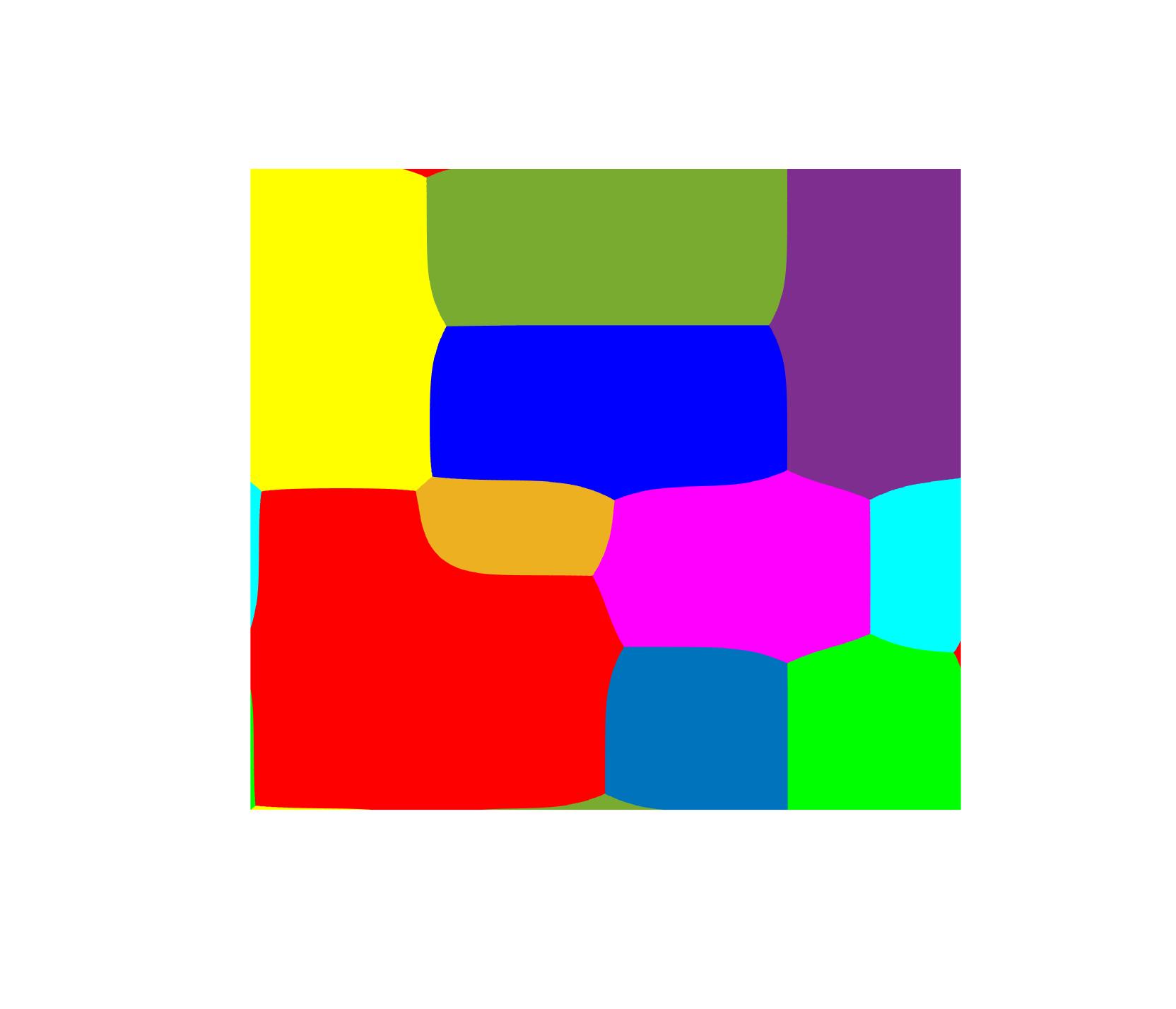}}
\subfigure[t=0.5]{
\includegraphics[width=0.30\textwidth,clip==]{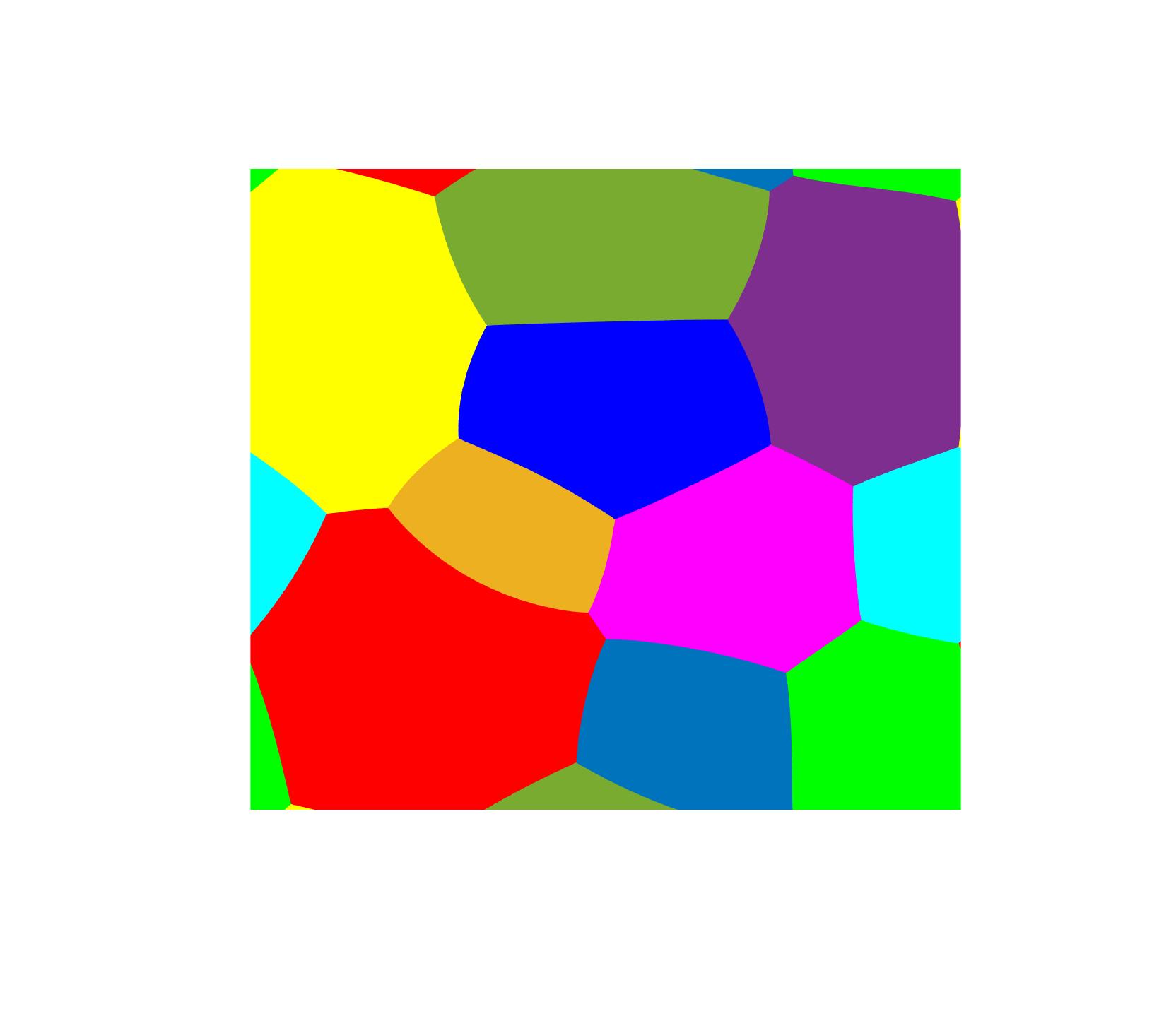}}
\subfigure[t=2]{
\includegraphics[width=0.30\textwidth,clip==]{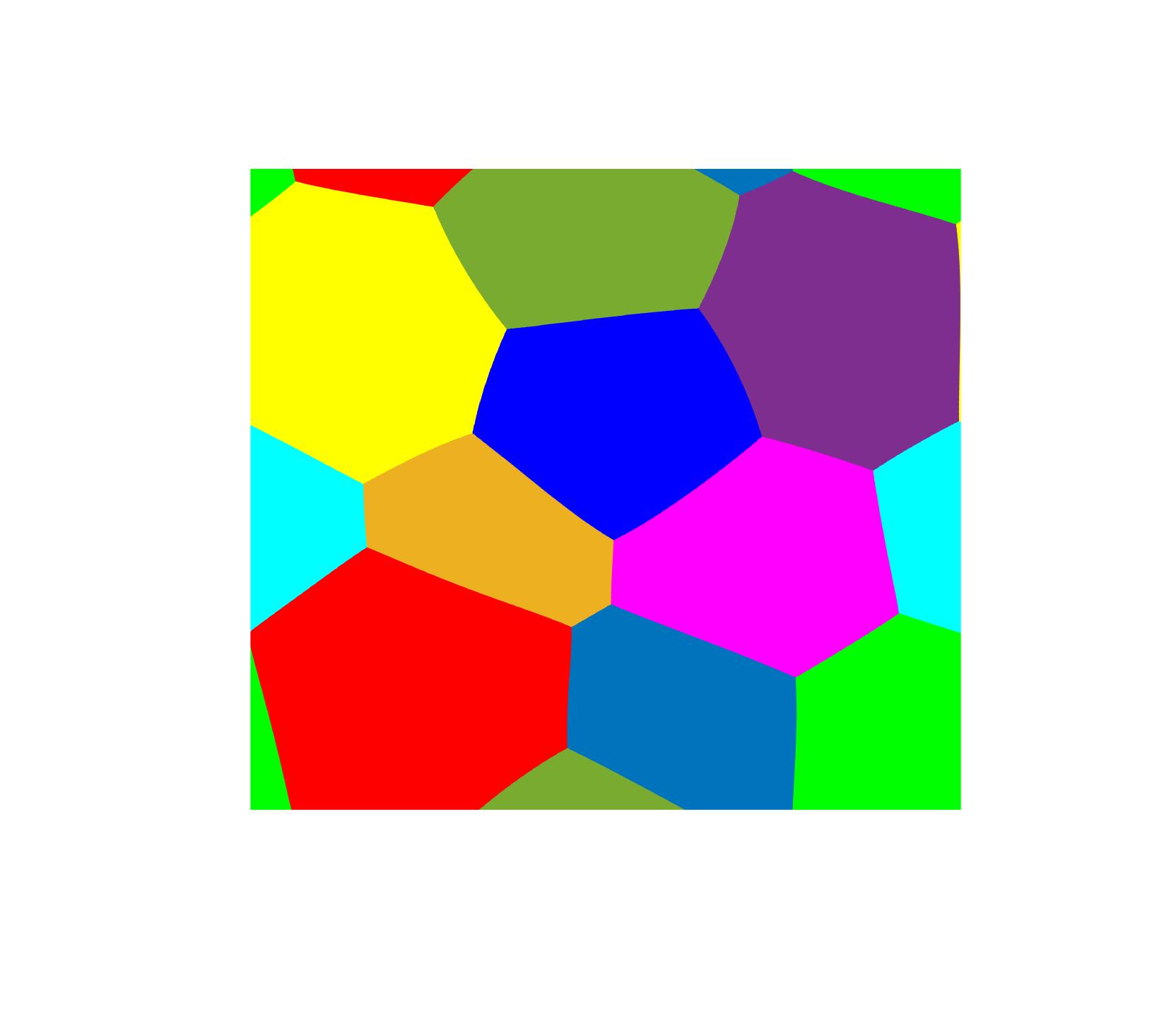}}
\subfigure[t=5]{
\includegraphics[width=0.30\textwidth,clip==]{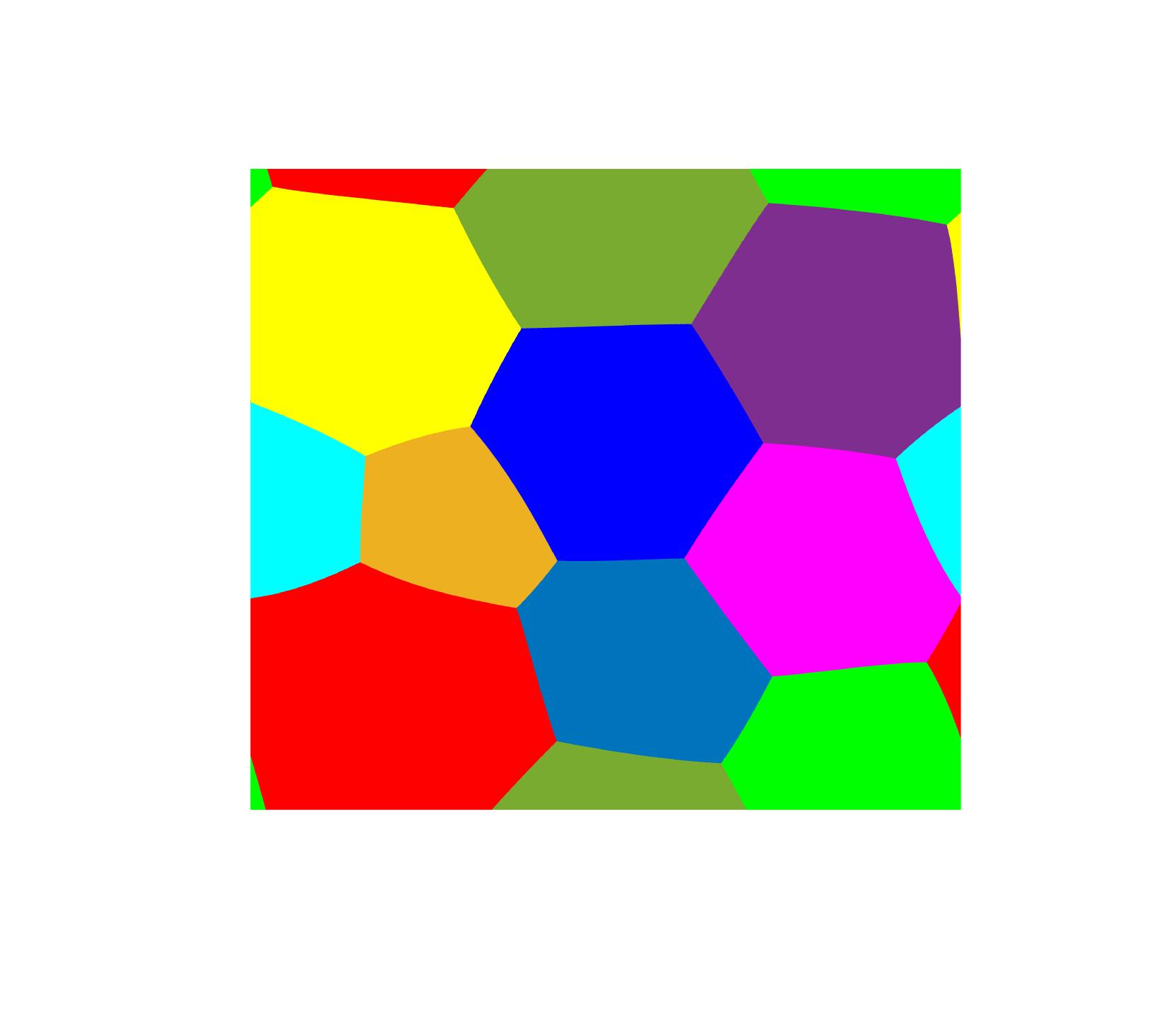}}
\subfigure[t=10]{
\includegraphics[width=0.30\textwidth,clip==]{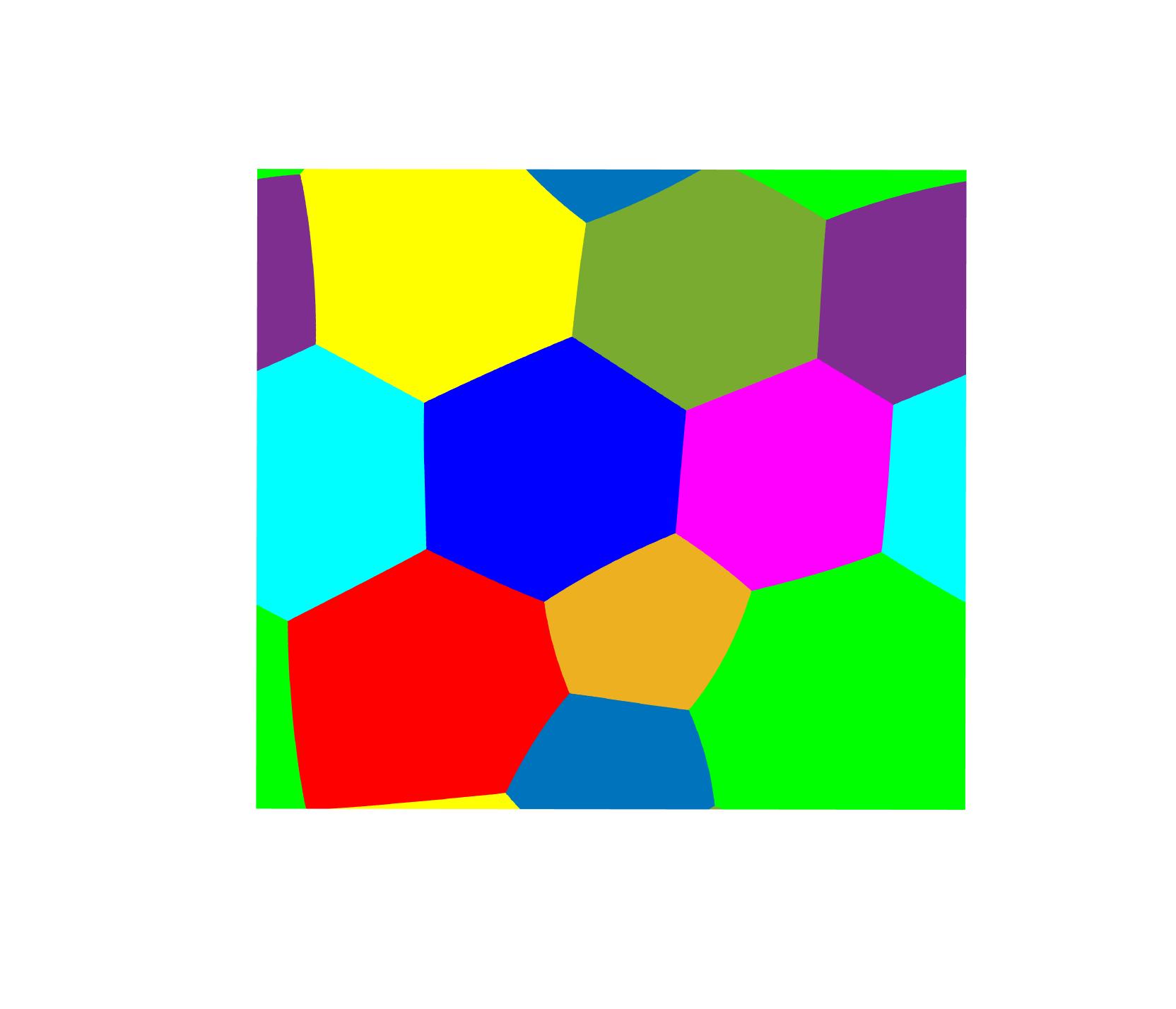}}
\caption{A 10-subdomain partition: initial partition and subdomains at times $t=0,0.05,0.5,2,5,10$ computed by the BDF2 scheme of first approach  with $\delta t=1\times 10^{-5}$.}\label{LGM-partition10-2D}
\end{figure}

\section{Concluding remarks}

How to construct efficient numerical schemes for gradient flows with global constraints is a challenging task. The popular penalty approach may lead to very stiff systems that are difficult to solve,  while a direct implementation of Lagrangian multiplier approach leads to nonlinear systems to solve at each time step.
We developed several efficient numerical schemes which can preserve exactly the constraints  for gradient flows with global constraints by combining  the SAV approach  with the Lagrangian multiplier approach.  These schemes are as efficient as the SAV schemes for unconstrained gradient flows,  i.e., only require solving linear  equations with constant coefficients at each time step plus an additional nonlinear algebraic system which can be solved at negligible cost,  and preserve exactly the constraints  for constrained gradient flows.  Moreover,  the second and third approaches lead to   schemes which are unconditionally energy stable.  And the Lagrangian multipliers in the third approach can be determined sequentially,  as opposed to coupled together in the second approach,  making it more robust and efficient than the second approach.

 We presented ample numerical results to compare the three approaches  with  the penalty approach.  Our numerical results indicate that  the proposed approaches can achieve accurate results and  preserve exactly the constraints with larger time steps than those allowed in the penalty approach.   And  the first and  third approaches  are more  robust and efficient than second approach.

Although we considered only time-discretization schemes in this paper,  they can be combined with any consistent finite dimensional Galerkin type approximations in practice,    since the stability  proofs are all based on  variational  formulations with all test functions in the same space as the trial functions.

\bibliographystyle{siamplain}
\bibliography{references}
\end{document}